\documentclass{amsart}

\newtheorem{theorem}{Theorem}[section]
\newtheorem{lemma}[theorem]{Lemma}
\newtheorem{proposition}[theorem]{Proposition}
\newtheorem{corollary}[theorem]{Corollary}
\newtheorem{remark}[theorem]{Remark}
\theoremstyle{definition}

\numberwithin{equation}{section}

\begin{document}

\title{On an extension of the $H^{k}$ mean curvature flow}

\author{Yi Li }
\address{Department of Mathematics, Harvard University, Cambridge, MA 70803}
%\curraddr{Department of Mathematics and Statistics,
%Case Western Reserve University, Cleveland, Ohio 43403}
\email{yili@math.harvard.edu}

%\author{Author Two}
%\address{Mathematical Research Section, School of Mathematical Sciences,
%Australian National University, Canberra ACT 2601, Australia}

%    General info
%\subjclass[2000]{Primary 54C40, 14E20; Secondary 46E25, 20C20}

%\date{January 1, 2001 and, in revised form, June 22, 2001.}


%\keywords{Differential geometry, algebraic geometry}

\begin{abstract}
In this note we generalize an extension theorem in \cite{Le-Sesum}
and \cite{Xu-Ye-Zhao} of the mean curvature flow to the $H^{k}$ mean
curvature flow under some extra conditions. The main difficult problem in
proving the extension theorem is to find a suitable version of
Michael-Simon inequality for the $H^{k}$ mean curvature flow, and to
do a suitable Moser iteration process. These two problems are
overcame by imposing some extra conditions which may be weakened or
removed in our forthcoming paper \cite{Ostergaard-Li}. On the other
hand, we derive some estimates for the generalized mean curvature
flow, which have their own interesting.
\end{abstract}

\maketitle

\section{Introduction}

Let $M$ be a compact $n$-dimensional hypersurface without boundary,
which is smoothly embedded into the $(n+1)$-dimensional Euclidean
space $\mathbb{R}^{n+1}$ by the map
\begin{equation}
F_{0}: M\to\mathbb{R}^{n+1}.
\end{equation}
The generalized mean curvature flow (GMCF), an evolution equation of
the mean curvature $H(\cdot,t)$, is a smooth family of immersions
$F(\cdot,t): M\to\mathbb{R}^{n+1}$ given by
\begin{equation}
\frac{\partial}{\partial t}F(\cdot,t)=-f(H(\cdot,t))\nu(\cdot,t), \
\ \ \ \ \ F(\cdot,0)=F_{0}(\cdot),
\end{equation}
where $f: \mathbb{R}\to\mathbb{R}$ is a smooth function, depending
only on $H(\cdot,t)$, with some properties to guarantee  the short
time existence, and $\nu(\cdot,t)$ is the outer unit normal on
$M_{t}:=F(M,t)$ at $F(\cdot,t)$. The short time existence of the
GMCF has been established in \cite{Smoczyk}. Namely, if $f'>0$ along
the GMCF, then it always admits a smooth solution on a maximal time
interval $[0,T_{{\max}})$ with $T_{{\rm max}}<\infty$. Setting $f$
the identity function is the classical mean curvature flow; on the
other hand, if we choose $f(x)$ to be some power function $x^{k}$, then
one gets the $H^{k}$ mean curvature flow. In this note we mainly
focus on the $H^{k}$ mean curvature flow, but partly results on the
GMCF are also derived.

In general, Huisken \cite{Huisken} proved that the mean curvature
flow develops to singularities in finite time: Suppose that $T_{{\rm
max}}<\infty$ is the first singularity time for the mean curvature
flow. Then $\sup_{M_{t}}|A|(t)\to\infty$ as $t\to T_{{\rm max}}$.

Recently, Cooper \cite{Cooper}, Le-Sesum \cite{Le-Sesum}, and
Xu-Ye-Zhao \cite{Xu-Ye-Zhao} proved an extension theorem on the mean
curvature flow under some curvature conditions. A natural question
is whether we can generalize it to the GMCF, in particular, the
$H^{k}$ mean curvature flow. In this note, we give a partial answer
to this question.

\begin{theorem} Suppose that the integers $n$ and $k$ are greater than or equal to $2$ and that $n+1\geq k$.
Suppose that $M$ is a compact $n$-dimensional hypersurface
without boundary, smoothly embedded into $\mathbb{R}^{n+1}$ by the
function $F_{0}$. Consider the $H^{k}$ mean curvature flow on $M$
\begin{equation*}
\frac{\partial}{\partial t}F(\cdot, t)=-H^{k}(\cdot,t)\nu(\cdot,t),
\ \ \ F(\cdot,0)=F_{0}(\cdot).
\end{equation*}
If

\begin{itemize}

\item[(a)] $h_{ij}(t)\geq Cg_{ij}(t)$ along the $H^{k}$ mean curvature flow for an uniform constant $C>0$,

\item[(b)] for some $\alpha\geq n+k+1$,
\begin{equation*}
\|H(t)\|_{L^{\alpha}(M\times[0,T_{{\rm max}}))}:=\left(\int^{T_{{\rm
max}}}_{0}\int_{M_{t}}|H(t)|^{\alpha}_{g(t)}d\mu(t)dt\right)^{\frac{1}{\alpha}}<\infty,
\end{equation*}
\end{itemize}
then the flow can be extended over the time $T_{{\rm max}}$.
\end{theorem}

\begin{remark}
When $k=1$, $n+1\geq k$ is trivial and the condition (a) should be
weaken to be $h_{ij}(t)\geq-Cg_{ij}(t)$ for some uniform constant
$C>0$ (see \cite{Le-Sesum} and \cite{Xu-Ye-Zhao}). we don't know the
condition $n+1\geq k$ is necessary, but in this note it is a
technique assumption when we use the similar method in
\cite{Le-Sesum}. In the forthcoming paper \cite{Ostergaard-Li}, we
want to at least weaken the condition (a) and to remove the
assumption $n+1\geq k$.
\end{remark}

For the generalized mean curvature flow, we have the following two
interesting estimates.

\begin{theorem} Suppose that the integers $n$ and $k$ are greater than or equal to $2$. Suppose that $M$ is a compact $n$-dimensional hypersurface
without boundary, smoothly embedded into $\mathbb{R}^{n+1}$ by the
function $F_{0}$. Consider the GMCF
\begin{equation*}
\frac{\partial}{\partial t}F(\cdot,t)=-f(H(\cdot,t))\nu(\cdot,t), \
\ \ F(\cdot,0)=F_{0}(\cdot), \ \ \ 0\leq t\leq T\leq T_{{\rm max}}<\infty.
\end{equation*}
Suppose that $f\in C^{\infty}(\Omega)$ for an open set
$\Omega\subset\mathbb{R}$, and that $v$ is a smooth function on
$M\times[0,T]$ such that its image is contained in $\Omega$.
Consider the differential inequality
\begin{equation}
\left(\frac{\partial}{\partial t}-\Delta_{f,t}\right)v\leq G\cdot
f(v)+f''(v)|\nabla_{t}v|^{2}_{g(t)}, \ \ \ v\geq0, \ \ \ G\in
L^{q}(M\times[0,T]).
\end{equation}
Let
\begin{equation*}
C_{0,q}=\|f'(v)G\|_{L^{q}(M\times[0,T])}, \ \ \
C_{1}=\left(1+\|H\|^{n+k+1}_{L^{n+k+1}(M\times[0,T])}\right)^{\frac{1}{k}},
\end{equation*}
and also let
\begin{equation*}
\gamma=2+\frac{(k+1)^{2}}{k^{2}n}.
\end{equation*}
We denote by $\mathcal{S}$ the set of all functions $f\in
C^{\infty}(\Omega)$, where $\Omega\subset\mathbb{R}$ is the domain
of $f$, satisfying

\begin{itemize}

\item[(i)] $f$ satisfies the differential inequality (1.3),

\item[(ii)] $f'(x)>0$ for all $x\in\Omega$,

\item[(iii)] $f(x)\geq0$ whenever $x\geq0$,

\item[(iv)] $f(H(t))H(t)\geq0$ along the GMCF.

\item[(v)] $f'(v)\geq C_{2}>0$ on $M\times[0,T]$ for some uniform constant $C_{2}$.

\end{itemize}
For any $\beta\geq2$ and $q>\frac{\gamma}{\gamma-2}$, there exists a
positive constant $C_{n,k,T}(C_{0,q},C_{1},\beta,q)$, depending only
on $n,k,T,\beta,q$, $C_{0,q}$,  $C_{1}$, and ${\rm Vol}(M)$, such that, for any
$f\in\mathcal{S}$,
\begin{eqnarray*}
& & \|\eta^{2}f^{\beta}(v)\|_{L^{\gamma/2}(M\times[0,T])}
\\
&\leq&C_{n,k,T}(C_{0,q},C_{1},\beta,q)\left\|f^{\beta}(v)\left[\eta^{2}+2\eta\left(\frac{\partial}{\partial
t}-f'(v)\Delta_{t}\right)\eta\right.\right.
\\
&+&\left.\left.\left(\frac{1}{\beta}\frac{f(v)f''(v)}{f'(v)}+\frac{8\beta^{2}-2\beta+2}{\beta(\beta-1)}f'(v)\right)
|\nabla_{t}\eta|^{2}_{g(t)}\right]\right\|_{L^{1}(M\times[0,T])}
\end{eqnarray*}
where (the definition of $B_{n,k,T}$ is given in Section 3)
\begin{equation*}
C_{n,k,T}(C_{0,q},C_{1},\beta,q)=\frac{\beta}{\beta-1}\max\left\{2(\widetilde{B}_{n,k,T}C_{1})^{2/\gamma},\left(2C_{0,q}\frac{\beta^{2}}{\beta-1}(\widetilde{B}_{n,k,T}C_{1})^{2/\gamma}\right)^{1+\nu}\right\},
\end{equation*}
$\nu=\frac{\gamma}{(\gamma-2)q-\gamma}$, and $\eta$ is any smooth
function on $M\times[0,T]$ with the property that $\eta(x,0)=0$ for
all $x\in M$. In particular, if $f'(v)G\in
L^{\infty}(M\times[0,T])$, then, letting $q\to\infty$, we have
\begin{eqnarray*}
C_{n,k,T}(C_{0,\infty},C_{1},\beta,\infty)&=&\frac{2\beta}{\beta-1}\max\left\{1,\frac{C_{0,\infty}\beta^{2}}{\beta-1}\right\}(\widetilde{B}_{n,k,T}C_{1})^{2/\gamma}
\\
&\leq&\left[8\max\{1,C_{0,\infty}\}\widetilde{B}_{n,k,T}^{2/\gamma}\right]\beta
C^{2/\gamma}_{1},
\end{eqnarray*}
where
\begin{eqnarray*}
\widetilde{B}_{n,k,T}=B_{n,k,T}\cdot\max\left\{\left(\frac{1}{C_{2}}\right)^{\frac{k+1}{2k}},1\right\},
\ \ \ C_{0,\infty}=\|f'(v)G\|_{L^{\infty}(M\times[0,T])},
\end{eqnarray*}
since $\frac{\beta}{\beta-1}\leq2$; in this case, we obtain
\begin{eqnarray*}
& & \|\eta^{2}f^{\beta}(v)\|_{L^{\gamma/2}(M\times[0,T])}
\\
&\leq&D_{n,k,T}\beta
C^{2/\gamma}_{1}\left\|f^{\beta}(v)\left[\eta^{2}+2\eta\left(\frac{\partial}{\partial
t}-f'(v)\Delta_{t}\right)\eta\right.\right.
\\
&+&\left.\left.\left(\frac{1}{\beta}\frac{f(v)f''(v)}{f'(v)}+\frac{8\beta^{2}-2\beta+2}{\beta(\beta-1)}f'(v)\right)
|\nabla_{t}\eta|^{2}_{g(t)}\right]\right\|_{L^{1}(M\times[0,T])},
\end{eqnarray*}
where
$D_{n,k,T}=8\max\{1,C_{0,\infty}\}\widetilde{B}_{n,k,T}^{2/\gamma}$.
\end{theorem}

\begin{corollary} Suppose that the integers $n$ and $k$ are greater than or equal to $2$. Suppose that $M$ is a compact $n$-dimensional hypersurface
without boundary, smoothly embedded into $\mathbb{R}^{n+1}$ by the
function $F_{0}$. Consider the GMCF
\begin{equation*}
\frac{\partial}{\partial t}F(\cdot,t)=-f(H(\cdot,t))\nu(\cdot,t), \
\ \ F(\cdot,0)=F_{0}(\cdot), \ \ \ 0\leq t\leq T\leq T_{{\rm max}}<\infty.
\end{equation*}
Suppose that $f\in C^{\infty}(\Omega)$ for an open set
$\Omega\subset\mathbb{R}$, and that $v$ is a smooth function on
$M\times[0,T]$ such that its image is contained in $\Omega$.
Consider the differential inequality
\begin{equation}
\left(\frac{\partial}{\partial t}-\Delta_{f,t}\right)v\leq G\cdot
f(v)+f''(v)|\nabla_{t}v|^{2}_{g(t)}, \ \ \ v\geq0, \ \ \ G\in
L^{q}(M\times[0,T]).
\end{equation}
Let
\begin{equation*}
C_{0,\infty}=\|f'(v)G\|_{L^{\infty}(M\times[0,T])}, \ \ \
C_{1}=\left(1+\|H\|^{n+k+1}_{L^{n+k+1}(M\times[0,T])}\right)^{\frac{1}{k}},
\end{equation*}
and also let
\begin{equation*}
\gamma=2+\frac{(k+1)^{2}}{k^{2}n}.
\end{equation*}
We denote by $\mathcal{S}$ the set of all functions $f\in
C^{\infty}(\Omega)$, where $\Omega\subset\mathbb{R}$ is the domain
of $f$, satisfying

\begin{itemize}

\item[(i)] $f$ satisfies the differential inequality (1.4),

\item[(ii)] $f'(x)>0$ for all $x\in\Omega$,

\item[(iii)] $f(x)\geq0$ whenever $x\geq0$,

\item[(iv)] $f(H(t))H(t)\geq0$ along the GMCF.

\item[(v)] $f'(v)\geq C_{2}>0$ on $M\times[0,T]$ for some uniform constant $C_{2}$.

\end{itemize}
There exists an uniform constant $C_{n}>0$, depending only on $n$,
such that for any $\beta\geq2$ and $f\in\mathcal{S}$ we have
\begin{equation*}
\|f(v)\|_{L^{\infty}\left(M\times\left[\frac{T}{2},T\right]\right)}\leq
E_{n,k,T}(\beta)\cdot
C^{\frac{1}{\beta}\frac{2}{\gamma-2}}_{1}\cdot\|f(v)\|_{L^{\beta}(M\times[0,T])},
\end{equation*}
where
\begin{equation*}
E_{n,k,T}(\beta)=(D_{n,k,T}C_{n}\beta)^{\frac{1}{\beta}\frac{\gamma}{\gamma-2}}\cdot\left(\frac{\gamma}{2}\right)^{\frac{1}{\beta}\frac{2\gamma}{(\gamma-2)^{2}}}\cdot
4^{\frac{1}{\beta}\frac{\gamma^{2}}{(\gamma-2)^{2}}},
\end{equation*}
and the constant $D_{n,k,T}$ is given in theorem 1.3.
\end{corollary}

{\bf Convention.} If $f(x): \mathbb{R}\to\mathbb{R}$ is a smooth
function, $v(t)$ is another smooth function, throughout this note we
denote by $f'(v)$ the value of $f'(x)$ at $x=v(t)$, namely,
\begin{equation*}
f'(v):=\frac{d}{dx}f(x)\big|_{x=v}.
\end{equation*}
When we write $\frac{d}{dt}f(v)$, it means that
\begin{equation*}
\frac{d}{dt}f(v(t))=\frac{d}{dx}f(x)\Big|_{x=v(t)}\cdot\frac{d}{dt}v(t)=f'(v(t))v'(t).
\end{equation*}
For example, if $f(x)=x^{k}$, then
\begin{equation*}
f'(v)=kv^{k-1}, \ \ \ \frac{d}{dt}f(v)=kv^{k-1}v'.
\end{equation*}

{\bf Acknowledgements.} The paper was written by author discussing
the Fourier analysis and number theory with Andrew Ostergaard who
asked him this problem. The author would particularly like to thank
Daniel Paradis for many discussions on complex analysis and
differential geometry during the summer.

\section{Evolution equations for GMCF}

In this section we fix our notation and derive some evolution
equations for the GMCF. Let $g=\{g_{ij}\}$ be the induced metric on
$M$ obtained by pullbacking the standard metric
$g_{\mathbb{R}^{n+1}}$ of $\mathbb{R}^{n+1}$. We denote by
$A=\{h_{ij}\}$ the second fundamental form and $d\mu=\sqrt{{\rm
det}(g_{ij})}$ the volume form on $M$, respectively. Using the local
coordinates system and above notation, the mean curvature can be
expressed as
\begin{equation}
H=g^{ij}h_{ij}.
\end{equation}
For any two mixed tensors, say $T=\{T^{i}_{jk}\}$ and
$S=\{S^{i}_{jk}\}$, their inner product relative to the induced
metric $g$ is given by
\begin{equation}
\langle
T^{i}_{jk},S^{i}_{jk}\rangle_{g}=g_{is}g^{jr}g^{ku}T^{i}_{jk}S^{s}_{ru}.
\end{equation}
Then the norm of the tensor $T$ is written as
\begin{equation}
|T|^{2}_{g}=\langle T^{i}_{jk},T^{i}_{jk}\rangle_{g}.
\end{equation}
Using  this notion, we have $|A|^{2}_{g}=g^{ij}g^{kl}h_{ik}h_{jl}$.
If $x_{1},\cdots,x_{n}$ are local coordinates on $M$, one has
\begin{equation}
g_{ij}=\left\langle\frac{\partial F}{\partial x_{i}},\frac{\partial
F}{\partial x_{j}}\right\rangle_{\mathbb{R}^{n+1}}, \ \ \ \ \ \
h_{ij}=-\left\langle\nu,\frac{\partial^{2}F}{\partial x_{i}\partial
x_{j}}\right\rangle_{\mathbb{R}^{n+1}},
\end{equation}
where $\langle\cdot,\cdot\rangle_{\mathbb{R}^{n+1}}$ denotes the
Euclidean inner product of $\mathbb{R}^{n+1}$. Let $\nabla$ denote
the induced Levi-Civita connection on $M$. Hence for an vector
$X=\{X^{i}\}$ we have
\begin{equation}
\nabla_{j}X^{i}=\frac{\partial}{\partial
x_{j}}X^{i}+\Gamma^{i}_{jk}X^{k},
\end{equation}
where $\Gamma^{i}_{jk}$ is the Christoffel symbol locally given by
\begin{equation}
\Gamma^{k}_{ij}=\frac{1}{2}g^{kl}\left(\frac{\partial
g_{jl}}{\partial x_{i}}+\frac{\partial g_{il}}{\partial
x_{j}}-\frac{\partial g_{ij}}{\partial x_{l}}\right).
\end{equation}
The induced Laplacian operator $\Delta$ on $M$ is defined by
\begin{equation}
\Delta T^{i}_{jk}:=g^{mn}\nabla_{m}\nabla_{n}T^{i}_{jk}.
\end{equation}
Moreover, the Laplacian operator $\Delta h_{ij}$ can be written as
\begin{equation}
\Delta
h_{ij}=\nabla_{i}\nabla_{j}H+Hh_{il}g^{lm}h_{mj}-|A|^{2}_{g}h_{ij}.
\end{equation}

We write $g(t)=\{g_{ij}(t)\}, A(t)=\{h_{ij}(t)\}, \nu(t), H(t),
d\mu(t), \nabla_{t}$, and $\Delta_{t}$ the corresponding induced
metric, second fundamental form, outer unit normal vector, mean
curvature, volume form, induced Levi-Civita connection, and induced
Laplacian operator at time $t$. The position coordinates are not
explicitly written in the above symbols if there is no confusion.

\begin{proposition} {\bf (Evolution equations)} For the GMCF, one has
\begin{eqnarray*}
\frac{\partial}{\partial t}F(t)&=&-f(H(t))\nu(t), \\
\frac{\partial}{\partial t}g_{ij}(t)&=&\nabla_{t}f(H(t))=f'(H(t))\cdot\nabla_{t}H(t), \\
\frac{\partial}{\partial t}h_{ij}(t)&=&f'(H(t))\cdot\Delta_{t}h_{ij}(t)+f''(H(t))\nabla_{i}H\cdot\nabla_{j}H(t) \\
&-&[f(H(t))+f'(H(t))H(t))]h_{il}(t)g^{lm}(t)h_{mj}(t)+f'(H(t))|A(t)|^{2}_{g(t)}h_{ij}(t), \\
\frac{\partial}{\partial t}H(t)&=&f'(H(t))\Delta_{t}H(t)+f(H(t))|A(t)|^{2}_{g(t)}+f''(H(t))|\nabla_{t}H(t)|^{2}_{g(t)}, \\
\frac{\partial}{\partial t}d\mu(t)&=&-f(H(t))H(t)d\mu(t).
\end{eqnarray*}
\end{proposition}

\begin{proof} The proof is straightforward, but is more tedious than that in the classical setting.\end{proof}

From the  evolution equation for the mean curvature $H(t)$, it is
natural to introduce the generalized Laplacian operator associated
to the function $f$. Put
\begin{equation}
\Delta_{f,t}(\cdot):=f'(\cdot)\Delta_{t}(\cdot).
\end{equation}
Hence
\begin{equation}
\frac{\partial}{\partial
t}H(t)=\Delta_{f,t}H(t)+f(H(t))|A(t)|^{2}_{g(t)}+f''(H(t))|\nabla_{t}H(t)|^{2}_{g(t)}.
\end{equation}
It is a special case of the following differential inequality
\begin{equation}
\left(\frac{\partial}{\partial t}-\Delta_{f,t}\right)v\leq G\cdot
f(v)+f''(v)|\nabla_{t}v|^{2}_{g(t)},
\end{equation}
which is also discussed in \cite{PLax}.

\section{A version of Michael-Simon inequality}

Let us consider that $M$  is the standard sphere $S^{n}$ which is
immersed into $\mathbb{R}^{n+1}$ by $F_{0}$. Just as in Example 2.1
\cite{Le-Sesum}, the $H^{k}$ mean curvature flow with initial data
$F_{0}$ has the formula $F(t)=r(t)F_{0}$. Hence
\begin{equation*}
\frac{dr(t)}{dt}=-\frac{n^{k}}{r^{k}(t)}, \ \ \ r(0)=1.
\end{equation*}
This ODE gives $r(t)=[1-(k+1)n^{k}t]^{\frac{1}{k+1}}$. The maximal
time is $T_{{\rm max}}=\frac{1}{(k+1)n^{k}}$. Using $T_{{\rm max}}$
we can rewrite $r(t)$ as
\begin{equation*}
r(t)=[(k+1)n^{k}(T_{{\rm max}}-t)]^{\frac{1}{k+1}}.
\end{equation*}
Hence the $L^{\alpha}$-norm of $H(t)$ on $M\times[0,T]$ is
\begin{eqnarray*}
\|H(t)\|^{\alpha}_{L^{\alpha}(M\times[0,T_{{\rm
max}}))}=\frac{n^{\alpha}\omega_{n}}{[(k+1)n^{k}]^{\frac{\alpha-n}{k+1}}}\int^{T_{{\rm
max}}}_{0}\frac{dt}{(T-t)^{\frac{\alpha-n}{k+1}}},
\end{eqnarray*}
which is finite if $\alpha<n+k+1$. Here $\omega_{n}$ denotes the
area of $S^{n}$. It implies that the constant $\alpha$ in Theorem
1.1 is optional. When $\alpha=n+k+1$, we consider a rescaling
transformation
\begin{equation*}
\widetilde{F}(\cdot,t)=Q^{\beta}F\left(\cdot,\frac{t}{Q^{\gamma}}\right).
\end{equation*}
In order to make sure that $\|H(t)\|_{L^{n+k+1}(M\times[0,T_{{\rm
max}}))}$ is invariant under this transformation, we must have
\begin{equation*}
\gamma=\beta(k+1).
\end{equation*}
In particular, $\|H(t)\|_{L^{n+k+1}(M\times[0,T_{{\rm max}}))}$ is
invariant under the following rescaling transformation
\begin{equation}
\widetilde{F}(\cdot,t)=Q\cdot F\left(\cdot,\frac{t}{Q^{k+1}}\right).
\end{equation}

\begin{remark} In general, we consider the rescaling transformation
of the GMCF
\begin{equation*}
\widetilde{F}(\cdot,t)=Q^{\beta}F\left(\cdot,\frac{t}{Q^{\gamma}}\right).
\end{equation*}
In order to guarantee that the quantity
$\|H(t)\|_{L^{\alpha}(M\times[0,T_{{\rm max}}))}$ is invariant under
this rescaling, we must have, for any $x$ and $Q>0$,
\begin{equation*}
\gamma=(\alpha-n)\beta, \ \ \
f(x)=Q^{\gamma-\beta}f\left(\frac{x}{Q^{\beta}}\right).
\end{equation*}
Letting $k=\alpha-n-1$, we obtain
\begin{equation}
f(x)=Q^{k\beta}f\left(\frac{x}{Q^{\beta}}\right), \ \ \
x\in\mathbb{R}, \ \ \ Q>0.
\end{equation}
A solution for this functional equation is $f(x)=x^{k}$. Actually,
we can show that the functional equation (3.2) has the unique
solution with the form $f(x)=f(1)x^{k}$. Indeed\footnote{Andrew told
me this short proof.}, if we let $y=1/Q$, then
\begin{equation*}
y^{k\beta}f(x)=f(xy^{\beta});
\end{equation*}
putting $x=1$ gives $f(y^{\beta})=f(1)y^{\beta}$ and hence
$f(x)=f(1)x^{k}$. This is a reason why we restrict ourself to the
$H^{k}$ mean curvature flow.
\end{remark}

The key step in \cite{Le-Sesum} is to establish a version of
Michael-Simon inequality. When $k=1$, this type of equality has been
proved in \cite{Le-Sesum}. Considering the $H^{k}$ mean curvature
flow, one should generalize the Michael-Simon inequality to a
"nonlinear" version when $k\geq2$. The first trying step is how to
find a suitable "nonlinear" number $Q$ satisfying the property that
it reduces to the original definition (that is, $Q=\frac{n}{n-2}$)
when $k$ equals $1$. There are lots of such choices on this step,
for instance, $Q=\frac{n}{n-k-1}, \frac{kn}{kn-2},
\frac{kn}{kn-(k+1)}$, etc. The first two numbers are easily to think
about, but the third one is not so easily to find out, since there
are at least two rules to obey: one should be compatible with the
H\"older's inequality, Young's inequality, and interpolation
inequality in the process of the proof; the second one is that we
should find an analogous inequality which is the original one when
$k=1$.

\begin{remark} Here we give a heuristical proof why we chose $Q=\frac{kn}{kn-(k+1)}$. Starting from $w=v^{a}$ with
some constant $a$ determined later and using the original
Michael-Simon inequality (see below) we have (in the following
estimates we omit constants in each step)
\begin{equation*}
\left(\int_{M}v^{\frac{\alpha
n}{n-1}}d\mu\right)^{\frac{n-1}{n}}\leq\int_{M}\left(|\nabla
v|v^{a-1}+|H|v^{a}\right)d\mu.
\end{equation*}
From H\"older's inequality and Young's inequality, one has
\begin{eqnarray*}
\left(\int_{M}v^{\frac{an}{n-1}}d\mu\right)^{\frac{n-1}{an}\frac{1}{b}}&\leq&\left(\int_{M}(|\nabla v|v^{a-1}+|H|v^{a})d\mu\right)^{\frac{1}{ab}}, \\
&\leq&\|\nabla v\|^{\frac{1}{ab}}_{L^{p}(M)}\|v\|^{\frac{a-1}{ab}}_{L^{(a-1)q}(M)}+\|H\|^{\frac{1}{ab}}_{L^{r}(M)}\|v\|^{\frac{1}{b}}_{L^{as}(M)} \\
&\leq&\|v\|^{\frac{(a-1)\alpha}{ab}}_{L^{(a-1)q}(M)}+\|\nabla
v\|^{\frac{\beta}{ab}}_{L^{p}(M)}+\|H\|^{\frac{1}{ab}}_{L^{r}(M)}\|v\|^{\frac{1}{b}}_{L^{as}(M)},
\end{eqnarray*}
where we put the wight $\frac{1}{b}$ on both sides
(the reason will be seen soon), and
\begin{equation*}
\frac{1}{p}+\frac{1}{q}=\frac{1}{r}+\frac{1}{s}=\frac{1}{\alpha}+\frac{1}{\beta}=1,
\ \ \ \ \ \ p,q,r,s,\alpha,\beta>1.
\end{equation*}
We let
\begin{equation*}
\frac{1}{b}=\frac{(a-1)\alpha}{ab}, \ \ \ \ \ \
\frac{an}{n-1}=(a-1)q.
\end{equation*}
Therefore. $a=\frac{q(n-1)}{q(n-1)-n}$ and
$\alpha=\frac{q(n-1)}{n}$. Moreover
\begin{equation*}
\frac{an}{n-1}=\frac{qn}{(q-1)n-q}.
\end{equation*}
If $q=k+1$, then we get
\begin{equation*}
\frac{an}{n-1}=\frac{(k+1)n}{kn-(k+1)}=\frac{k+1}{k}\cdot\frac{kn}{kn-(k+1)}.
\end{equation*}
There are two reasons to set  $\frac{1}{b}=\frac{k+1}{k}$: the first
one comes from the careful investigation of the term
$\|H\|^{1/ab}_{L^{r}(M)}\|v\|^{1/b}_{L^{as}(M)}$ by  using the
interpolation inequality, and the another reason is the equation
$\frac{1}{c}+\frac{kn-(k+1)}{kn}=1$ which gives $c=\frac{kn}{k+1}$.
However, other reasons, e.g., $\frac{1}{p}+\frac{1}{k+1}=1$
determining $p=\frac{k+1}{k}$,  can be seen in the detailed analysis
of the proof.  The above is an exploration for finding a suitable
number $Q$, and, of course, is very naive and rough.
\end{remark}

Let $M$ be a compact $n$-dimensional hypersurface without boundary,
which is smoothly embedded in $\mathbb{R}^{n+1}$. The original
Michael-Simon inequality states that for any nonnegative,
$C^{1}$-functions $w$, one has
\begin{equation}
\left(\int_{M}w^{\frac{n}{n-1}}d\mu\right)^{\frac{n-1}{n}}\leq
c_{n}\int_{M}(|\nabla w|+|H|w)d\mu.\label{MSI}
\end{equation}
Here $c_{n}$ is the constant depending only on $n$. More precisely,
\begin{equation}
c_{n}=\frac{4^{n+1}}{\omega^{1/n}_{n}}, \ \ \ \ \ \
\omega_{n}={\rm Area}(S^{n}).
\end{equation}

Before proving the main theorem in this section, we state
some elementary integral inequalities which can be proven by
H\"older's inequality.

\begin{lemma} \label{EII} For any compact manifold $M$ and any Lipschitz functions $f$, one has
\begin{itemize}

\item[(i)] $\|f\|_{L^{p}(M)}\leq\|f\|_{L^{q}(M)}\cdot{\rm Vol}(M)^{\frac{q-p}{pq}}$ whenever $0<p<q$.

\item[(ii)] for any $k\geq1$, one has
\begin{equation*}
\int_{M}|f|^{1/k}d\mu\leq\left(\int_{M}|f|d\mu\right)^{1/k}\cdot{\rm
Vol}(M)^{\frac{k-1}{k}}.
\end{equation*}
\end{itemize}
Here $d\mu$ is the volume form of $M$ and ${\rm Vol}(M)$ is the
volume of $M$.
\end{lemma}

Also, we will use the inequalities (c.f. \cite{Grafakos})
\begin{eqnarray}
(a_{1}+a_{2})^{\theta}&\leq&a^{\theta}_{1}+a^{\theta}_{2}, \ \ \
0\leq\theta\leq1, \\
(a_{1}+a_{2})^{\theta}&\leq&2^{\theta-1}(a^{\theta}_{1}+a^{\theta}_{2}),
\ \ \ \theta\geq1,
\end{eqnarray}
where $a_{1}$ and $a_{2}$ are any nonnegative numbers.

\begin{theorem} Suppose that $k,n\geq2$, or, $k=1$ and $n>2$. Set
\begin{equation}
Q_{k}=\frac{kn}{kn-(k+1)}=\frac{n}{n-\frac{k+1}{k}}.
\end{equation}
Let $M$ be a compact $n$-dimensional hypersurface without boundary,
which is smoothly embedded in $\mathbb{R}^{n+1}$. Then, for all
nonnegative Lipschitz functions $v$ on $M$, we have
\begin{eqnarray}
\left\|v\right\|^{k+1}_{L^{\frac{k+1}{k}Q_{k}}(M)}&\leq&
A_{n,k}\left(\left\|\nabla
v\right\|^{k+1}_{L^{\frac{k+1}{k}}(M)}+\left\|H\right\|^{n+k+1}_{L^{n+k+1}(M)}\left\|v\right\|^{k+1}_{L^{\frac{k+1}{k}}(M)}\right),
\\
&\leq&\widehat{A}_{n,k}\left(\left\|\nabla
v\right\|^{k+1}_{L^{2}(M)}+\left\|H\right\|^{n+k+1}_{L^{n+k+1}(M)}\left\|v\right\|^{k+1}_{L^{2}(M)}\right).
\end{eqnarray}
where $A_{n,k}$ and $\widehat{A}_{n,k}$ are constants explicitly
given by ($
c_{n,k}=c_{n}\cdot\frac{(k+1)(n-1)}{kn-(k+1)}$)
\begin{eqnarray*}
A_{n,k}&=&2^{\frac{(n-1)(k+1)(n+k+1)}{kn-(k+1)}}(2c_{n,k})^{n+k+1}
\\
\widehat{A}_{n,k}&=&A_{n,k}\cdot{\rm Vol}(M)^{\frac{k-1}{2(k+1)}}.
\end{eqnarray*}

\end{theorem}

\begin{proof} The proof is quite similar to that given in \cite{Le-Sesum}.  The case that $k=1$ and $n>2$ has been proved in \cite{Le-Sesum},
hence we may assume that $k,n\geq2$. Let
\begin{equation*}
w=v^{\frac{(k+1)(n-1)}{kn-(k+1)}}.
\end{equation*}
Plugging it into (\ref{MSI}),we have
\begin{eqnarray*}
\left(\int_{M}v^{\frac{n(k+1)}{kn-(k+1)}}d\mu\right)^{\frac{n-1}{n}}&\leq&c_{n}\int_{M}\left(\frac{(k+1)(n-1)}{kn-(k+1)}|\nabla v|v^{\frac{n}{kn-(k+1)}}+|H|v^{\frac{(k+1)(n-1)}{kn-(k+1)}}\right)d\mu \\
&\leq&c_{n,k}\left(\int_{M}|\nabla
v|v^{\frac{n}{kn-(k+1)}}d\mu+\int_{M}|H|v^{\frac{(k+1)(n-1)}{kn-(k+1)}}d\mu\right),
\end{eqnarray*}
where
\begin{equation*}
c_{n,k}:=c_{n}\cdot\frac{(k+1)(n-1)}{kn-(k+1)}> c_{n}.
\end{equation*}
If we let
$a_{n,k}=[c_{n,k}]^{\frac{kn-(k+1)}{n-1}}\cdot2^{\frac{kn-k-n}{n-1}}$,
then, using H\"older's inequality and the inequality (3.4), one
concludes that (since $kn\geq k+n$)
\begin{eqnarray*}
& & \left(\int_{M}v^{\frac{(k+1)n}{kn-(k+1)}}d\mu\right)^{\frac{kn-(k+1)}{n}} \\
&\leq&[c_{n,k}]^{\frac{kn-(k+1)}{n-1}}\left(\int_{M}|\nabla v|v^{\frac{n}{kn-(k+1)}}d\mu+\int_{M}|H|v^{\frac{(k+1)(n-1)}{kn-(k+1)}}d\mu\right)^{\frac{kn-(k+1)}{n-1}} \\
&\leq&a_{n,k}\left(\|\nabla
v\|^{\frac{kn-(k+1)}{n-1}}_{L^{\frac{k+1}{k}}(M)}\|v\|^{\frac{n}{n-1}}_{L^{\frac{(k+1)n}{kn-(k+1)}}(M)}+\|H|^{\frac{kn-(k+1)}{n-1}}_{L^{r}(M)}\|v\|^{k+1}_{L^{\frac{(k+1)(n-1)}{kn-(k+1)}s}(M)}\right)
\end{eqnarray*}
where $r,s$ are positive real numbers satisfying
$\frac{1}{r}+\frac{1}{s}=1$. Recall Young's inequality
\begin{equation*}
ab\leq\epsilon a^{p}+\epsilon^{-q/p}b^{q},
\end{equation*}
where $a,b,\epsilon>0$, $p,q>1$, and $\frac{1}{p}+\frac{1}{q}=1$.
Putting
\begin{equation*}
p=\frac{(k+1)(n-1)}{n}, \ \ \ \ \ \ q=\frac{(k+1)(n-1)}{kn-(k+1)}, \
\ \ \ \ \ \frac{p}{q}=\frac{kn-(k+1)}{n},
\end{equation*}
we derive that, for any $\epsilon>0$,
\begin{equation*}
\|\nabla
v\|^{\frac{kn-(k+1)}{n-1}}_{L^{\frac{k+1}{k}}(M)}\|v\|^{\frac{n}{n-1}}_{L^{\frac{(k+1)n}{kn-(k+1)}}(M)}\leq\epsilon\|v\|^{k+1}_{L^{\frac{(k+1)n}{kn-(k+1)}}(M)}+\epsilon^{-\frac{n}{kn-(k+1)}}\|\nabla
v\|^{k+1}_{L^{\frac{k+1}{k}}(M)}.
\end{equation*}
There is a natural way to find a suitable value of $s$, when we use
the interpolation inequality to bound the first term appeared above
using $L^{\frac{k+1}{k}}$-norm and
$L^{\frac{(k+1)n}{kn-(k+1)}}$-norm. Suppose now that
\begin{equation}
\frac{kn-k-1}{kn-k}<1<s<\frac{n}{n-1}. \label{VOS}
\end{equation}
According to (\ref{VOS}), we must have
\begin{equation*}
\frac{k+1}{k}<\frac{(k+1)(n-1)}{kn-(k+1)}s<\frac{(k+1)n}{kn-(k+1)}.
\end{equation*}
Applying the interpolation inequality to our case gives
\begin{equation*}
\|v\|_{L^{\frac{(k+1)(n-1)}{kn-(k+1)}s}(M)}\leq\delta\|v\|_{L^{\frac{(k+1)n}{kn-(k+1)}}(M)}+\delta^{-\mu}\|v\|_{L^{\frac{k+1}{k}}(M)},
\ \ \ \delta>0,
\end{equation*}
where the constant $\mu$ is determined by
\begin{equation*}
\mu=\frac{\frac{k}{k+1}-\frac{kn-(k+1)}{(k+1)(n-1)s}}{\frac{kn-(k+1)}{(k+1)(n-1)s}-\frac{kn-(k+1)}{(k+1)n}}=\frac{n}{kn-(k+1)}\cdot\frac{k(n-1)(s-1)+1}{n-(n-1)s}:=\mu_{n,k,s}.
\end{equation*}
Thus, together with Jensen's inequality, we yield
\begin{eqnarray*}
& & \|v\|^{k+1}_{L^{\frac{(k+1)n}{kn-(k+1)}}(M)}  \ \leq \  a_{n,k}\left(\epsilon\|v\|^{k+1}_{L^{\frac{(k+1)n}{kn-(k+1)}}(M)}+\epsilon^{-\frac{n}{kn-k-1}}\|\nabla v\|^{k+1}_{L^{\frac{k+1}{k}}(M)}\right. \\
&+&\left.2^{k}\|H\|^{\frac{kn-(k+1)}{n-1}}_{L^{r}(M)}\left(\delta^{k+1}\|v\|^{k+1}_{L^{\frac{(k+1)n}{kn-(k+1)}}(M)}+\left(\delta^{k+1}\right)^{-\mu_{n,k,s}}\|v\|^{k+1}_{L^{\frac{k+1}{k}}(M)}\right)\right).
\end{eqnarray*}
Simplifying above implies that
\begin{eqnarray*}
& & \left(1-\epsilon\cdot a_{n,k}-2^{k}a_{n,k}\delta^{k+1}\|H\|^{\frac{kn-(k+1)}{n-1}}_{L^{r}(M)}\right)\|v\|^{k+1}_{L^{\frac{(k+1)n}{kn-(k+1)}}(M)} \\
&\leq&a_{n,k}\epsilon^{-\frac{n}{kn-(k+1)}}\|\nabla
v\|^{k+1}_{L^{\frac{k+1}{k}}(M)}+2^{k}a_{n,k}\left(\delta^{k+1}\right)^{-\mu_{n.k.s}}\|H\|^{\frac{kn-(k+1)}{n-1}}_{L^{r}(M)}\|v\|^{k+1}_{L^{\frac{k+1}{k}}(M)}.
\end{eqnarray*}
Let (Here, we may assume that $\|H\|_{L^{r}(M)}\neq0$; otherwise it
is trivial.)
\begin{equation*}
\epsilon=\frac{1}{2a_{n,k}}, \ \ \
\delta^{k+1}=\frac{1}{2^{k+2}a_{n,k}}\|H\|^{-\frac{kn-(k+1)}{n-1}}_{L^{r}(M)}.
\end{equation*}
Therefore, we have (note that $\frac{1}{r}+\frac{1}{s}=1$)
\begin{eqnarray*}
\|v\|^{k+1}_{L^{\frac{(k+1)n}{kn-(k+1)}}(M)}&\leq&2(2a_{n,k})^{\frac{(k+1)(n-1)}{kn-(k+1)}}\|\nabla v\|^{k+1}_{L^{\frac{k+1}{k}}(M)} \\
&+&\left(2^{2+k}a_{n,k}\right)^{\frac{n-1}{kn-(k+1)}\cdot\frac{(k+1)r}{r-n}}\|H\|^{\frac{(k+1)r}{r-n}}_{L^{r}(M)}\|v\|^{k+1}_{L^{\frac{k+1}{k}}(M)}.
\end{eqnarray*}
The condition (3.9) turns out $r>n$. Setting
\begin{equation*}
\frac{(k+1)r}{r-n}=r
\end{equation*}
gives us $r=n+k+1$ which is our required result. Plugging the
explicit formula for $a_{n,k}$ in terms of $c_{n,k}$ into above and
using Lemma \ref{EII}, we obtain
\begin{eqnarray*}
\|v\|^{k+1}_{L^{\frac{k+1}{k}}(M)}&\leq&2(2c_{n,k})^{k+1}\|\nabla
v\|^{k+1}_{L^{\frac{k+1}{k}}(M)}
\\
&+&2^{\frac{(n-1)(k+1)(n+k+1)}{kn-(k+1)}}(2c_{n,k})^{n+k+1}\|H\|^{n+k+1}_{L^{n+k+1}(M)}\|v\|^{k+1}_{L^{\frac{k+1}{k}}(M)}.
\end{eqnarray*}
Noting that the coefficient appeared in the first term is less than
that in the second term, we obtain the inequality.
\end{proof}

\begin{corollary} Under the condition of Theorem 3.3, for any nonnegative
Lipschitz functions $v$, we have
\begin{eqnarray*}
\|v\|^{2}_{L^{2Q_{k}}(M)}\leq
\widetilde{A}_{n,k}\left(\|v\|^{\frac{k-1}{k}}_{L^{2}(M)}\cdot\|\nabla
v\|^{\frac{k+1}{k}}_{L^{2}(M)}+\left(\|H\|^{n+k+1}_{L^{n+k+1}(M)}\right)^{1/k}\|v\|^{2}_{L^{2}(M)}\right),
\end{eqnarray*}
where the uniform constant $\widetilde{A}_{n,k}$ is given by
\begin{equation*}
\widetilde{A}_{n,k}=A^{1/k}_{n,k}\cdot\left(\frac{2k}{k+1}\right)^{\frac{k+1}{k}}.
\end{equation*}
\end{corollary}

\begin{proof} Replacing $v$ by $v^{\frac{2k}{k+1}}$ in Theorem 3.3,
we obtain
\begin{eqnarray*}
\|v\|^{2k}_{2Q_{k}}&\leq&A_{n,k}\left(\left\|\frac{2k}{k+1}\cdot
v^{\frac{k-1}{k+1}}\cdot\nabla
v\right\|^{k+1}_{L^{\frac{k+1}{k}}(M)}+\|H\|^{n+k+1}_{L^{n+k+1}(M)}\cdot\|v^{\frac{2k}{k+1}}\|^{k+1}_{L^{\frac{k+1}{k}}(M)}\right)
\\
&=&A_{n,k}\left(\left\|\left(\frac{2k}{k+1}\right)^{\frac{k+1}{k}}v^{\frac{k-1}{k}}(\nabla
v)^{\frac{k+1}{k}}\right\|^{k}_{L^{1}(M)}+\|H\|^{n+k+1}_{L^{n+k+1}(M)}\cdot\|v\|^{2k}_{L^{2}(M)}\right)
\\
&\leq&A_{n,k}\left(\left(\frac{2k}{k+1}\right)^{k+1}\|v^{\frac{k-1}{k}}\|^{k}_{L^{\frac{2k}{k-1}}(M)}\|(\nabla
v)^{\frac{k+1}{k}}\|^{k}_{L^{\frac{2k}{k+1}}(M)}\right)
\\
&+&A_{n,k}\|H\|^{n+k+1}_{L^{n+k+1}(M)}\|v\|^{2k}_{L^{2}(M)} \\
&\leq&A_{n,k}\left(\left(\frac{2k}{k+1}\right)^{k+1}\|v\|^{k-1}_{L^{2}(M)}\|\nabla
v\|^{k+1}_{L^{2}(M)}+\|H\|^{n+k+1}_{L^{n+k+1}(M)}\|v\|^{2k}_{L^{2}(M)}\right).
\end{eqnarray*}
Taking the $k$th root on both sides gives the required inequality.
\end{proof}

\begin{theorem} Let $n$ and $k$ are integers bigger than or equal to $2$. Consider the GMCF
\begin{equation*}
\frac{\partial}{\partial t}F(\cdot,t)=-f(H(\cdot,t))\nu(\cdot,t), \
\ \ 0\leq t\leq T\leq T_{{\rm max}}<\infty,
\end{equation*}
where $f\in C^{\infty}(\Omega)$ is a smooth function over an open
set $\Omega\subset\mathbb{R}$. Suppose that $f'(x)>0$ and $f(x)\cdot
x\geq0$ along the GMCF. For all nonnegative Lipschitz functions $v$,
one has
\begin{eqnarray*}
\|v\|^{\beta}_{L^{\beta}(M\times[0,T])}&\leq&B_{n,k,T}\cdot\max_{0\leq t\leq T}\|v\|^{\frac{(k+1)^{2}}{k^{2}n}+\frac{k-1}{k}}_{L^{2}(M_{t})} \\
&\cdot&\left(\|\nabla_{t}
v\|^{\frac{k+1}{k}}_{L^{2}(M\times[0,T])}\right. \\
&+&\left.\max_{0\leq t\leq
T}\|v\|^{\frac{k+1}{k}}_{L^{2}(M_{t})}\cdot\left(\|H\|^{n+k+1}_{L^{n+k+1}(M\times[0,T])}\right)^{\frac{1}{k}}\right),
\end{eqnarray*}
where $B_{n,k,T}$ is the constant explicitly given by
\begin{equation*}
B_{n,k,T}=\widetilde{A}_{n,k}\cdot{\rm
Vol}(M)^{\frac{(k-1)(k+1)}{2k^{2}n}}\cdot\max\left\{T^{\frac{k-1}{k}},T^{\frac{k-1}{2k}}\right\}
\end{equation*}
and $\beta=2+\frac{k+1}{k}\cdot\frac{k+1}{kn}>2$.
\end{theorem}
\begin{proof}  Setting $p=\frac{kn}{kn-(k+1)}$ and $q=\frac{kn}{k+1}$ in H\"older's inequality, we have
\begin{eqnarray*}
\|v\|^{\beta}_{L^{\beta}(M\times[0,T])}&=&\int^{T}_{0}dt\int_{M_{t}}v^{2}\cdot v^{\frac{k+1}{k}\cdot\frac{k+1}{kn}}d\mu(t) \\
&\leq&\int^{T}_{0}dt\left(\int_{M_{t}}v^{2Q_{k}}d\mu(t)\right)^{1/Q_{k}}\left(\int_{M_{t}}v^{\frac{k+1}{k}}d\mu(t)\right)^{\frac{k+1}{kn}} \\
&=&\max_{0\leq t\leq
T}\|v\|^{\frac{(k+1)^{2}}{k^{2}n}}_{L^{\frac{k+1}{k}}(M_{t})}\cdot\int^{T}_{0}\|v\|^{2}_{L^{2Q_{k}}(M_{t})}dt.
\end{eqnarray*}
The assumption $f(x)\cdot x\geq0$ implies that
\begin{equation*}
\frac{d}{dt}\mu(t)=-f(H(t))\cdot H(t)\mu(t)\leq0,
\end{equation*}
consequently, the volume is deceasing along the GMCF. This fact
combining with Lemma 3.2 gives
\begin{eqnarray*}
\max_{0\leq t\leq
T}\|v\|^{\frac{(k+1)^{2}}{k^{2}n}}_{L^{\frac{k+1}{k}}(M_{t})}&\leq&\max_{0\leq
t\leq T}\left(\|v\|_{L^{2}(M_{t})}\cdot{\rm
Vol}(M_{t})^{\frac{k-1}{2(k+1)}}\right)^{\frac{(k+1)^{2}}{k^{2}n}} \\
&\leq&\max_{0\leq t\leq
T}\|v\|^{\frac{(k+1)^{2}}{k^{2}n}}_{L^{2}(M_{t})}\cdot{\rm
Vol}(M)^{\frac{(k-1)(k+1)}{2k^{2}n}}.
\end{eqnarray*}
On other hand, we have
\begin{eqnarray*}
\int^{T}_{0}\|v\|^{2}_{L^{2Q_{k}}(M_{t})}dt&\leq&\widetilde{A}_{n,k}\int^{T}_{0}\left(\|
v\|^{\frac{k-1}{k}}_{L^{2}(M_{t})}\cdot\|\nabla_{t}
v\|^{\frac{k+1}{k}}_{L^{2}(M_{t})}\right.
\\
&+&\left.\|v\|^{2}_{L^{2}(M_{t})}\left(\|H\|^{n+k+1}_{L^{n+k+1}(M_{t})}\right)^{\frac{1}{k}}\right)dt
\\
&\leq&\widetilde{A}_{n,k}\cdot\max_{0\leq t\leq
T}\|v\|^{\frac{k-1}{k}}_{L^{2}(M_{t})}\cdot\int^{T}_{0}\|\nabla_{t}
v\|^{\frac{k+1}{k}}_{L^{2}(M_{t})}dt
\\
&+&\widetilde{A}_{n,k}\cdot\max_{0\leq t\leq
T}\|v\|^{2}_{L^{2}(M_{t})}\cdot\int^{T}_{0}\left(\|H\|^{n+k+1}_{L^{n+k+1}(M_{t})}\right)^{1/k}dt.
\end{eqnarray*}
From Lemma 3.2, we obtain
\begin{eqnarray*}
\int^{T}_{0}\left(\|H\|^{n+k+1}_{L^{n+k+1}(M_{t})}\right)^{1/k}dt&\leq&\left(\int^{T}_{0}\|H\|^{n+k+1}_{L^{n+k+1}(M_{t})}dt\right)^{1/k}T^{\frac{k-1}{k}},
\\
&=&\left(\|H\|^{n+k+1}_{L^{n+k+1}(M\times[0,T])}\right)^{1/k}\cdot T^{\frac{k-1}{k}}, \\
\int^{T}_{0}\|\nabla
v\|^{\frac{k+1}{k}}_{L^{2}(M_{t})}dt&=&\int^{T}_{0}\left(\|\nabla_{t}v\|^{2}_{L^{2}(M_{t})}\right)^{\frac{1}{2k/(k+1)}}dt
\\
&\leq&\|\nabla_{t}v\|^{\frac{k+1}{k}}_{L^{2}(M\times[0,T])}\cdot
T^{\frac{k-1}{2k}}.
\end{eqnarray*}
Plugging it into above inequality, one yields
\begin{eqnarray*}
\|v\|^{\beta}_{L^{\beta}(M\times[0,T])}&\leq&\max_{0\leq t\leq
T}\|v\|^{\frac{(k+1)^{2}}{k^{2}n}}_{L^{2}(M_{t})}\cdot ({\rm
Vol}(M))^{\frac{(k-1)(k+1)}{2k^{2}n}}\cdot\widetilde{A}_{n,k} \\
&\cdot&\max_{0\leq t\leq
T}\|v\|^{\frac{k-1}{k}}_{L^{2}(M_{t})}\cdot\max\left\{T^{\frac{k-1}{k}},T^{\frac{k-1}{2k}}\right\}
\\
&\cdot&\left(\|\nabla_{t}v\|^{\frac{k+1}{k}}_{L^{2}(M_{t})}\right.
\\
&+&\left.\max_{0\leq t\leq
T}\|v\|^{\frac{k+1}{k}}_{L^{2}(M_{t})}\left(\|H\|^{n+k+1}_{L^{n+k+1}(M\times[0,T])}\right)^{1/k}\right),
\end{eqnarray*}
which is the required result.
\end{proof}

\begin{remark} If $k=1$, then $\frac{k+1}{k}=2$; hence we do not need to use Lemma 3.2
to control the terms by $L^{2}$-norm and carefully checking the proof gives
$B_{n,1,T}=A_{n,1}$, which is the constant derived in
\cite{Le-Sesum}.

\end{remark}

\section{Moser iteration for the $H^{k}$ mean curvature flow}

In this section we generalize Lemma 4.1 in \cite{Le-Sesum} to the
GMCF, in particular, to the $H^{k}$ mean curvature flow. The proof
is similar to that given in \cite{Le-Sesum}, but it doesn't directly
follow words by words from \cite{Le-Sesum} since the differential
inequality now involves an extra term $f''(v)|\nabla v|^{2}$. When
$f(x)=x^{k}$ and $k=1$, that is, the classical mean curvature flow,
this term automatically vanishes. Since the mean curvature $H(t)$
along the generalized mean curvature flow satisfies
\begin{equation*}
\frac{\partial}{\partial
t}H(t)=f'(H(t))\Delta_{t}H(t)+f(H(t))|A(t)|^{2}_{g(t)}+f''(H(t))|\nabla_{t}H(t)|^{2}_{g(t)},
\end{equation*}
we should study the differential inequality
\begin{equation}
\left(\frac{\partial}{\partial t}-\Delta_{f,t}\right)v\leq G\cdot
f(v)+f''(v)|\nabla_{t} v|^{2}_{g(t)}, \ \ \ v\geq0, \ \ \ G\in
L^{q}(M\times[0,T]). \label{DI}
\end{equation}
Let $\eta(x,t)$ be any smooth function on $M\times[0,T]$ with the
property that $\eta(x,0)=0$ for all $x\in M$.

Later, we will chose $\eta(x,t)$ to be a smooth function only
relative to the variable $t$, satisfying the above property, and
$f(x)=x^{k}$.

\begin{theorem} Suppose that the integers $n$ and $k$ are greater than or equal to $2$. Consider the GMCF
\begin{equation*}
\frac{\partial}{\partial t}F(\cdot,t)=-f(H(\cdot,t))\nu(\cdot,t), \
\ \ 0\leq t\leq T\leq T_{{\rm max}}<\infty.
\end{equation*}
Suppose that $f\in C^{\infty}(\Omega)$ for an open set
$\Omega\subset\mathbb{R}$, and that $v$ is a smooth function on
$M\times[0,T]$ such that its image is contained in $\Omega$.
Consider the differential inequality
\begin{equation}
\left(\frac{\partial}{\partial t}-\Delta_{f,t}\right)v\leq G\cdot
f(v)+f''(v)|\nabla_{t}v|^{2}, \ \ \ v\geq0, \ \ \ G\in
L^{q}(M\times[0,T]).
\end{equation}
Let
\begin{eqnarray*}
C_{0,q}&=&\|f'(v)G\|_{L^{q}(M\times[0,T])},
\\
C_{1}&=&\left(1+\|H\|^{n+k+1}_{L^{n+k+1}(M\times[0,T])}\right)^{\frac{1}{k}},
\end{eqnarray*}
and also let
\begin{equation*}
\gamma=2+\frac{(k+1)^{2}}{k^{2}n}.
\end{equation*}
We denote by $\mathcal{S}$ the set of all functions $f\in
C^{\infty}(\Omega)$, where $\Omega\subset\mathbb{R}$ is the domain
of $f$, satisfying

\begin{itemize}

\item[(i)] $f$ satisfies the differential inequality (4.2),

\item[(ii)] $f'(x)>0$ for all $x\in\Omega$,

\item[(iii)] $f(x)\geq0$ whenever $x\geq0$,

\item[(iv)] $f(H(t))H(t)\geq0$ along the GMCF.

\item[(v)] $f'(v)\geq C_{2}>0$ on $M\times[0,T]$ for some uniform constant $C_{2}$.

\end{itemize}
For any $\beta\geq2$ and $q>\frac{\gamma}{\gamma-2}$, there exists a
positive constant $C_{n,k,T}(C_{0,q},C_{1},\beta,q)$, depending only
on $n,k,T,\beta,q$, $C_{0,q}$, $C_{1}$, and ${\rm Vol}(M)$, such
that, for any $f\in\mathcal{S}$,
\begin{eqnarray*}
& & \|\eta^{2}f^{\beta}(v)\|_{L^{\gamma/2}(M\times[0,T])}
\\
&\leq&C_{n,k,T}(C_{0,q},C_{1},\beta,q)\left\|f^{\beta}(v)\left[\eta^{2}+2\eta\left(\frac{\partial}{\partial
t}-f'(v)\Delta_{t}\right)\eta\right.\right.
\\
&+&\left.\left.\left(\frac{1}{\beta}\frac{f(v)f''(v)}{f'(v)}+\frac{8\beta^{2}-2\beta+2}{\beta(\beta-1)}f'(v)\right)
|\nabla_{t}\eta|^{2}_{g(t)}\right]\right\|_{L^{1}(M\times[0,T])}
\end{eqnarray*}
where
\begin{eqnarray*}
& & C_{n,k,T}(C_{0,q},C_{1},\beta,q)
\\
&=&\frac{\beta}{\beta-1}\max\left\{2(B_{n,k,T}C_{1})^{2/\gamma},\left(2C_{0,q}\frac{\beta^{2}}{\beta-1}(B_{n,k,T}C_{1})^{2/\gamma}\right)^{1+\nu}\right\},
\end{eqnarray*}
$\nu=\frac{\gamma}{(\gamma-2)q-\gamma}$, and $\eta$ is any smooth
function on $M\times[0,T]$ with the property that $\eta(x,0)=0$ for
all $x\in M$. In particular, if $f'(v)G\in
L^{\infty}(M\times[0,T])$, then, letting $q\to\infty$, we have
\begin{eqnarray*}
& & C_{n,k,T}(C_{0,\infty},C_{1},\beta,\infty)
\\
&=&\frac{2\beta}{\beta-1}\max\left\{1,\frac{C_{0,\infty}\beta^{2}}{\beta-1}\right\}(\widetilde{B}_{n,k,T}C_{1})^{2/\gamma}
\\
&\leq&\left[8\max\{1,C_{0,\infty}\}\widetilde{B}_{n,k,T}^{2/\gamma}\right]\beta
C^{2/\gamma}_{1},
\end{eqnarray*}
where
\begin{eqnarray*}
\widetilde{B}_{n,k,T}&=&B_{n,k,T}\cdot\max\left\{\left(\frac{1}{C_{2}}\right)^{\frac{k+1}{2k}},1\right\},
\\
C_{0,\infty}&=&\|f'(v)G\|_{L^{\infty}(M\times[0,T])},
\end{eqnarray*}
since $\frac{\beta}{\beta-1}\leq2$; in this case, we obtain
\begin{eqnarray*}
& & \|\eta^{2}f^{\beta}(v)\|_{L^{\gamma/2}(M\times[0,T])}
\\
&\leq&D_{n,k,T}\beta
C^{2/\gamma}_{1}\left\|f^{\beta}(v)\left[\eta^{2}+2\eta\left(\frac{\partial}{\partial
t}-f'(v)\Delta_{t}\right)\eta\right.\right.
\\
&+&\left.\left.\left(\frac{1}{\beta}\frac{f(v)f''(v)}{f'(v)}+\frac{8\beta^{2}-2\beta+2}{\beta(\beta-1)}f'(v)\right)
|\nabla_{t}\eta|^{2}_{g(t)}\right]\right\|_{L^{1}(M\times[0,T])},
\end{eqnarray*}
where
$D_{n,k,T}=8\max\{1,C_{0,\infty}\}\widetilde{B}_{n,k,T}^{2/\gamma}$.
\end{theorem}

\begin{remark} The set $\mathcal{S}$, in general, may not be empty.
For example, let $v(\cdot,t)=H(\cdot,t)\geq0$ and suppose that
$f(x)=x^{k}$, $\Omega=\mathbb{R}^{+}$, and $f'(H(t))\geq C_{2}>0$
along the GMCF; we immediately see that the conditions (ii) (iii),
and (v) are satisfied. For (iv),
\begin{equation*}
f(H(t))H(t)=H^{k+1}(t)=H^{k-1}(t)\cdot H^{2}(t)\geq0.
\end{equation*}
This will be applied to our case.
\end{remark}

\begin{proof} Applying  the test function $\eta^{2}f'(v)f^{\beta-1}(v)$ to our differential inequality (4.1), for any $s\in[0,T]$, we have
\begin{eqnarray*}
& & \int^{s}_{0}\int_{M_{t}}(-\Delta_{f,t}v)\eta^{2}f'(v)f^{\beta-1}(v)d\mu(t)dt \\
&+&\int^{s}_{0}\int_{M_{t}}\frac{\partial v}{\partial t}\eta^{2}f'(v)f^{\beta-1}(v)d\mu(t)dt \\
&\leq&\int^{s}_{0}\int_{M_{t}}|G|\eta^{2}f'(v)f^{\beta}(v)d\mu(t)dt
\\
&+&\int^{s}_{0}\int_{M_{t}}\eta^{2}f'(v)f''(v)f^{\beta-1}(v)|\nabla_{t}
v|^{2}_{g(t)}d\mu(t)dt.
\end{eqnarray*}
Integrating by parts gives
\begin{eqnarray*}
& & \int_{M_{t}}(-\Delta_{f,t}v)\eta^{2}f'(v)f^{\beta-1}(v)d\mu(t)dt
\ = \
\int_{M_{t}}(-\Delta_{t}v)\eta^{2}(f'(v))^{2}f^{\beta-1}(v)d\mu(t)
\\
&=&\int_{M_{t}}\langle\nabla_{t}v,\nabla_{t}(\eta^{2}(f'(v))^{2}f^{\beta-1}(v))\rangle_{g(t)}
d\mu(t) \\
&=&\int_{M_{t}}\langle\nabla_{t}v,2\nabla_{t}\eta\cdot\eta(f'(v))^{2}f^{\beta-1}(v)\rangle_{g(t)}d\mu(t)
\\
&+&\int_{M_{t}}\langle\nabla_{t}v,\eta^{2}(2f'(v)f''(v)f^{\beta-1}(v)\nabla_{t}v+(f'(v))^{3}(\beta-1)f^{\beta-2}(v)\nabla_{t}v)\rangle_{g(t)}d\mu(t)
\\
&=&2\int_{M_{t}}\langle\nabla_{t}v,\nabla_{t}\eta\rangle_{g(t)}\eta(f'(v))^{2}f^{\beta-1}(v)d\mu(t)
\\
&+&\int_{M_{t}}\eta^{2}[2f'(v)f''(v)f^{\beta-1}(v)+(\beta-1)(f'(v))^{3}f^{\beta-2}(v)]|\nabla_{t}v|^{2}_{g(t)}d\mu(t).
\end{eqnarray*}
Recall the evolution equation for volume form
\begin{equation*}
\frac{\partial}{\partial r}d\mu(t)=-f(H(t))\cdot H(t)\cdot d\mu(t).
\end{equation*}
Hence
\begin{eqnarray*}
& & \int^{s}_{0}\int_{M_{t}}\frac{\partial v}{\partial
t}\cdot\eta^{2}\cdot f'(v)f^{\beta-1}(v)d\mu(t)dt \\
&=&
\frac{1}{\beta}\int^{s}_{0}\int_{M_{t}}\frac{\partial(f^{\beta}(v))}{\partial
t}\eta^{2}d\mu(t)dt \\
&=&\frac{1}{\beta}\int_{M_{t}}f^{\beta}(v)\eta^{2}d\mu(t)\Big|^{s}_{0}-\frac{1}{\beta}\int^{s}_{0}\int_{M_{t}}f^{\beta}(v)\frac{\partial}{\partial
t}(\eta^{2}d\mu(t))dt \\
&=&\frac{1}{\beta}\int_{M_{s}}f^{\beta}(v)\eta^{2}d\mu(s)-\frac{1}{\beta}\int^{s}_{0}\int_{M_{t}}f^{\beta}(v)\left[2\eta\frac{\partial\eta}{\partial
t}-\eta^{2}f(H(t))H(t)\right]d\mu(t)dt.
\end{eqnarray*}
Combining these formulas and the assumption (iii), we conclude that
\begin{eqnarray*}
& &
\int^{s}_{0}\int_{M_{t}}\left[2\langle\nabla_{t}v,\nabla_{t}\eta\rangle_{g(t)}\eta(f'(v))^{2}f^{\beta-1}(v)\right.
\\
&+&\left.(2\eta^{2}f'(v)f''(v)f^{\beta-1}(v)+(\beta-1)\eta^{2}(f'(v))^{3}f^{\beta-1}(v))|\nabla_{t}v|^{2}_{g(t)}\right]d\mu(t)dt
\\
&+&\frac{1}{\beta}\int_{M_{s}}f^{\beta}(v)\eta^{2}d\mu(s) \\
&\leq&\frac{1}{\beta}\int^{s}_{0}\int_{M_{t}}f^{\beta}(v)\left[2\eta\frac{\partial\eta}{\partial
t}-\eta^{2}f(H(t))H(t)\right]d\mu(t)dt \\
&+&\int^{s}_{0}\int_{M_{t}}|G|\eta^{2}f'(v)f^{\beta}(v)d\mu(t)dt
\\
&+&\int^{s}_{0}\int_{M_{t}}\eta^{2}f'(v)f''(v)f^{\beta-1}(v)|\nabla_{t}v|^{2}_{g(t)}d\mu(t)dt
\\
&\leq&\frac{1}{\beta}\int^{s}_{0}\int_{M_{t}}f^{\beta}(v)2\eta\frac{\partial\eta}{\partial
t}d\mu(t)dt+\int^{s}_{0}\int_{M_{t}}|G|\eta^{2}f'(v)f^{\beta}(v)d\mu(t)dt
\\
&+&\int^{s}_{0}\int_{M_{t}}\eta^{2}f'(v)f''(v)f^{\beta-1}(v)|\nabla_{t}v|^{2}_{g(t)}d\mu(t)dt.
\end{eqnarray*}
Since
\begin{eqnarray*}
& &
\frac{1}{\beta}\int^{s}_{0}\int_{M_{t}}f^{\beta}(v)2\eta\frac{\partial\eta}{\partial
t}d\mu(t)dt \\
&=&\frac{1}{\beta}\int^{s}_{0}\int_{M_{t}}\left[f^{\beta}(v)2\eta\left(\frac{\partial}{\partial
t}-f'(v)\Delta_{t}\right)\eta+f^{\beta}(v)f'(v)2\eta\Delta_{t}\eta\right]d\mu(t)dt
\\
&=&\frac{1}{\beta}\int^{s}_{0}\int_{M_{t}}\left[f^{\beta}(v)2\eta\left(\frac{\partial}{\partial
t}-f'(v)\Delta_{t}\right)\eta-2\langle\nabla_{t}(f^{\beta}(v)f'(v)\eta),\nabla_{t}\eta\rangle_{g(t)}\right]d\mu(t)dt
\\
&=&\frac{1}{\beta}\int^{s}_{0}\int_{M_{t}}\left[f^{\beta}(v)2\eta\left(\frac{\partial}{\partial
t}-f'(v)\Delta_{t}\right)\eta-2\langle\beta
f^{\beta-1}(v)(f'(v))^{2}\eta\nabla_{t}v,\nabla_{t}\eta\rangle_{g(t)}\right.
\\
&-&2\langle f^{\beta}(v)(\eta
f''(v)\nabla_{t}v+f'(v)\nabla_{t}\eta),\nabla_{t}\eta\rangle_{g(t)}\Big]d\mu(t)dt
\\
&=&\frac{1}{\beta}\int^{s}_{0}\int_{M_{t}}f^{\beta}(v)\left[2\eta\left(\frac{\partial}{\partial
t}-f'(v)\Delta_{t}\right)\eta-2f'(v)|\nabla_{t}\eta|^{2}_{g(t)}\right]d\mu(t)dt
\\
&-&\frac{2}{\beta}\int^{s}_{0}\int_{M_{t}}\eta\left[\beta
f^{\beta-1}(v)(f'(v))^{2}+f^{\beta}(v)f''(v)\right]\langle\nabla_{t}v,\nabla_{t}\eta\rangle_{g(t)}d\mu(t)dt
\end{eqnarray*}
it follows that
\begin{eqnarray*}
& &
4\int^{s}_{0}\int_{M_{t}}\eta(f'(v))^{2}f^{\beta-1}(v)\langle\nabla_{t}v,\nabla_{t}\eta\rangle_{g(t)}d\mu(t)dt \ + \ \frac{1}{\beta}\int_{M_{s}}f^{\beta}(v)\eta^{2}d\mu(s)\\
&+&\int^{s}_{0}\int_{M_{t}}[(\beta-1)(f'(v))^{3}+f(v)f'(v)f''(v)]\eta^{2}f^{\beta-2}(v)|\nabla_{t}v|^{2}_{g(t)}d\mu(t)dt
\\
&\leq&\frac{1}{\beta}\int^{s}_{0}\int_{M_{t}}f^{\beta}(v)\left[2\eta\left(\frac{\partial}{\partial
t}-f'(v)\Delta_{t}\right)\eta-2f'(v)|\nabla_{t}\eta|^{2}_{g(t)}\right]d\mu(t)dt
\\
&+&\int^{s}_{0}\int_{M_{t}}|G|\eta^{2}f'(v)f^{\beta}(v)d\mu(t)dt
\\
&-&\frac{2}{\beta}\int^{s}_{0}\int_{M_{t}}\eta
f''(v)f^{\beta}(v)\langle\nabla_{t}v,\nabla_{t}\eta\rangle_{g(t)}d\mu(t)dt.
\end{eqnarray*}
The Cauchy-Schwartz inequality gives (where $\epsilon>0$)
\begin{eqnarray*}
& &
4\int^{s}_{0}\int_{M_{t}}\langle\nabla_{t}v,\nabla_{t}\eta\rangle_{g(t)}\eta(f'(v))^{2}f^{\beta-1}(v)d\mu(t)dt
\\
&\geq&-2\epsilon^{2}\int^{s}_{0}\int_{M_{t}}\eta^{2}(f'(v))^{3}f^{\beta-2}(v)|\nabla_{t}v|^{2}_{g(t)}d\mu(t)dt
\\
&-&\frac{2}{\epsilon^{2}}\int^{s}_{0}\int_{M_{t}}f'(v)f^{\beta}(v)|\nabla_{t}\eta|^{2}_{g(t)}d\mu(t)dt,
\end{eqnarray*}
and
\begin{eqnarray*}
& & \frac{2}{\beta}\int^{s}_{0}\int_{M_{t}}\eta
f''(v)f^{\beta}(v)\langle\nabla_{t}v,\nabla_{t}\eta\rangle_{g(t)}d\mu(t)dt
\\
&\geq&-\int^{s}_{0}\int_{M_{t}}f(v)f'(v)f''(v)f^{\beta-2}(v)\eta^{2}
|\nabla_{t}v|^{2}_{g(t)}d\mu(t)dt \\
&-&\frac{1}{\beta^{2}}\int^{s}_{0}\int_{M_{t}}\frac{f(v)f''(v)}{f'(v)}f^{\beta}(v)|\nabla_{t}\eta|^{2}_{g(t)}d\mu(t)dt.
\end{eqnarray*}
Consequently, we obtain
\begin{eqnarray*}
& &
\int^{s}_{0}\int_{M_{t}}[(\beta-1-2\epsilon^{2})f'(v)]\eta^{2}f^{\beta-2}(v)(f'(v))^{2}|\nabla_{t}v|^{2}_{g(t)}d\mu(t)dt \\
&+&\frac{1}{\beta}\int_{M_{s}}f^{\beta}(v)\eta^{2}d\mu(s) \\
&\leq&\frac{1}{\beta}\int^{s}_{0}\int_{M_{t}}f^{\beta}(v)\left[2\eta\left(\frac{\partial}{\partial
t}-f'(v)\Delta_{t}\right)\eta\right.
\\
&+&\left.\left(\frac{1}{\beta}\frac{f(v)f''(v)}{f'(v)}-2f'(v)+\frac{2\beta}{\epsilon^{2}}f'(v)\right)|\nabla_{t}\eta|^{2}_{g(t)}\right]d\mu(t)dt
\\
&+&\int^{s}_{0}\int_{M_{t}}|G|\eta^{2}f'(v)f^{\beta}(v)d\mu(t)dt.
\end{eqnarray*}
Note that
\begin{equation*}
|\nabla_{t}(f^{\beta/2}(v))|^{2}_{g(t)}=\frac{\beta^{2}}{4}f^{\beta-2}(v)(f'(v))^{2}|\nabla_{t}v|^{2}_{g(t)}.
\end{equation*}
If we choose $\beta-1=4\epsilon^{2}$, then the above inequality
gives us
\begin{eqnarray*}
& &
\frac{2(\beta-1)}{\beta}\int^{s}_{0}\int_{M_{t}}f'(v)\eta^{2}|\nabla_{t}(f^{\beta/2}(v))|^{2}_{g(t)}d\mu(t)dt+\int_{M_{s}}f^{\beta}(v)\eta^{2}d\mu(s)
\\
&\leq&\int^{s}_{0}\int_{M_{t}}f^{\beta}(v)\left[2\eta\left(\frac{\partial}{\partial
t}-f'(v)\Delta_{t}\right)\eta\right.
\\
&+&\left.\left(\frac{1}{\beta}\frac{f(v)f''(v)}{f'(v)}-2f'(v)+\frac{8\beta}{\beta-1}f'(v)\right)|\nabla_{t}\eta|^{2}_{g(t)}\right]d\mu(t)dt
\\
&+&\beta\int^{s}_{0}\int_{M_{t}}|G|\eta^{2}f'(v)f^{\beta}(v)d\mu(t)dt.
\end{eqnarray*}
Recall that
\begin{eqnarray*}
|\nabla_{t}(\eta f^{\beta/2}(v))|^{2}_{g(t)}&=&|\nabla_{t}\eta
\cdot f^{\beta/2}(v)+\eta\nabla_{t}(f^{\beta/2}(v))|^{2}_{g(t)} \\
&\leq&2\eta^{2}|\nabla_{t}(f^{\beta/2}(v))|^{2}_{g(t)}+2f^{\beta}(v)\cdot|\nabla_{t}\eta|^{2}_{g(t)}.
\end{eqnarray*}
Therefore
\begin{eqnarray*}
& & C_{2}\int^{s}_{0}\int_{M_{t}}|\nabla_{t}(\eta
f^{\beta/2})|^{2}_{g(t)}d\mu(t)dt+\int_{M_{s}}f^{\beta}(v)\eta^{2}d\mu(t)
\\
&\leq&\frac{\beta}{\beta-1}\int^{s}_{0}\int_{M_{t}}f^{\beta}(v)\left[2\eta\left(\frac{\partial}{\partial
t}-f'(v)\Delta_{t}\right)\eta\right. \\
&+&\left.\left(\frac{1}{\beta}\frac{f(v)f''(v)}{f'(v)}+\frac{8\beta^{2}-2\beta+2}{\beta(\beta-1)}f'(v)\right)|\nabla_{t}\eta|^{2}_{g(t)}\right]d\mu(t)dt
\\
&+&\frac{\beta^{2}}{\beta-1}\int^{s}_{0}\int_{M_{t}}|G|\eta^{2}f'(v)f^{\beta}(v)d\mu(t)dt
\\
&\leq&\frac{\beta}{\beta-1}\int^{s}_{0}\int_{M_{t}}f^{\beta}(v)\left[2\eta\left(\frac{\partial}{\partial
t}-f'(v)\Delta_{t}\right)\eta\right. \\
&+&\left.\left(\frac{1}{\beta}\frac{f(v)f''(v)}{f'(v)}+\frac{8\beta^{2}-2\beta+2}{\beta(\beta-1)}f'(v)\right)|\nabla_{t}\eta|^{2}_{g(t)}\right]d\mu(t)dt
\\
&+&\frac{\beta^{2}}{\beta-1}\|f'(v)G\|_{L^{q}(M\times[0,T])}\cdot\|\eta^{2}f^{\beta}\|_{L^{\frac{q}{q-1}}(M\times[0,T])}:=A.
\end{eqnarray*}
(In the following we also use the notion $\Lambda$ which is the
first term of $A$.) It gives us, for any $s$,
\begin{eqnarray*}
\|\eta f^{\beta/2}(v)\|_{L^{2}(M_{s})}&\leq& A^{1/2}, \\
\|\nabla_{t}(\eta f^{\beta/2}(v))\|_{L^{2}(M\times[0,T])}&\leq&
\left(\frac{A}{C_{2}}\right)^{1/2}.
\end{eqnarray*}
Using Theorem 3.6, one has
\begin{eqnarray*}
\|\eta f^{\beta/2}(v)\|^{\gamma}_{L^{\gamma}(M\times[0,T])}&\leq&
B_{n,k,T}\cdot\max_{0\leq s\leq T}\|\eta
f^{\beta/2}(v)\|^{\frac{(k+1)^{2}}{k^{2}n}+\frac{k-1}{k}}_{L^{2}(M_{s})}
\\
&\cdot&\left(\|\nabla_{t}(\eta
f^{\beta/2}(v))\|^{\frac{k+1}{k}}_{L^{2}(M\times[0,T])}\right.
\\
&+&\left.\max_{0\leq s\leq T}\|\eta
f^{\beta/2}(v)\|^{\frac{k+1}{k}}_{L^{2}(M_{s})}\cdot\left(\|H\|^{n+k+1}_{L^{n+k+1}(M\times[0,T])}\right)^{1/k}\right)
\\
&\leq&B_{n,k,T}\cdot\max\left\{\left(\frac{1}{C_{2}}\right)^{\frac{k+1}{2k}},1\right\}\cdot
A^{\frac{(k+1)^{2}}{2k^{2}n}+1} \\
&\cdot&\left[1+\left(\|H\|^{n+k+1}_{L^{n+k+1}(M\times[0,T])}\right)^{1/k}\right],
\\
&=&[\widetilde{B}_{n,k,T} C_{1}]\cdot
A^{\frac{(k+1)^{2}}{2k^{2}n}+1},
\end{eqnarray*}
where $\gamma=2+\frac{k+1}{k}\cdot\frac{k+1}{kn}$. Moreover,
\begin{eqnarray*}
\|\eta^{2}f^{\beta}\|_{L^{\gamma/2}(M\times[0,T])}&=&
\left(\|\eta f^{\beta/2}\|^{\gamma}_{L^{\gamma}(M\times[0,T])}\right)^{2/\gamma} \\
&\leq&A\cdot(\widetilde{B}_{n,k,T}C_{1})^{2/\gamma}
\\
&=&(\widetilde{B}_{n,k,T}C_{1})^{2/\gamma}\left(\Lambda+\frac{\beta^{2}}{\beta-1}C_{0}\|\eta^{2}f^{\beta}\|_{L^{\frac{q}{q-1}}(M\times[0,T])}\right),
\end{eqnarray*}
where $q>\frac{\gamma}{\gamma-2}$. Noting that
\begin{equation*}
1<\frac{q}{q-1}<\frac{\gamma}{2}
\end{equation*}
and using the interpolation inequality, one gets
\begin{equation*}
\|\eta^{2}f^{\beta}\|_{L^{\frac{q}{q-1}}(M\times[0,T])}\leq\epsilon\|\eta^{2}f^{\beta}\|_{L^{\gamma/2}(M\times[0,T])}+\epsilon^{-\nu}\|\eta^{2}f^{\beta}\|_{L^{1}(M\times[0,T])},
\end{equation*}
where the constant $\nu$ is defined by
\begin{equation*}
\nu=\frac{1-\frac{q-1}{q}}{\frac{q-1}{q}-\frac{2}{\gamma}}=\frac{\gamma}{(\gamma-2)q-\gamma}.
\end{equation*}
Therefore,
\begin{eqnarray*}
& & \|\eta^{2}f^{\beta}\|_{L^{\gamma/2}(M\times[0,T])}
\\
&\leq&\left[(\widetilde{B}_{n,k,T}C_{1})^{2/\gamma}\cdot\frac{\beta^{2}}{\beta-1}C_{0,q}\epsilon\right]\|\eta^{2}f^{\beta}\|_{L^{\gamma/2}(M\times[0,T])}
\\
&+&(\widetilde{B}_{n,k,T}C_{1})^{2/\gamma}\left(\Lambda+\frac{\beta^{2}}{\beta-1}C_{0,q}\epsilon^{-\nu}\|\eta^{2}f^{\beta}(v)\|_{L^{1}(M\times[0,T])}\right).
\end{eqnarray*}
If we chose
$(\widetilde{B}_{n,k,T}C_{1})^{2/\gamma}\cdot\frac{\beta^{2}}{\beta-1}\cdot
C_{0,q}\epsilon=\frac{1}{2}$, then
\begin{eqnarray*}
& & \|\eta^{2}f^{\beta}(v)\|_{L^{\gamma/2}(M\times[0,T])}
\\
&\leq&2(\widetilde{B}_{n,k,T}C_{1})^{2/\gamma}\Lambda
\\
&+&\left(2C_{0,q}\cdot\frac{\beta^{2}}{\beta-1}(\widetilde{B}_{n,k,T}C_{1})^{\frac{2}{\gamma}}\right)^{1+\nu}\|\eta^{2}f^{\beta}(v)\|_{L^{1}(M\times[0,T])}
\\
&\leq&\max\left\{2(\widetilde{B}_{n,k,T}C_{1})^{2/\gamma},\left(2C_{0,q}\cdot\frac{\beta^{2}}{\beta-1}(\widetilde{B}_{n,k,T}C_{1})^{\frac{2}{\gamma}}\right)^{1+\nu}\right\}
\\
&\cdot&\left(\Lambda+\|\eta^{2}f^{\beta}(v)\|_{L^{1}(M\times[0,T])}\right)
\\
&:=&\widetilde{C}_{n,k,T}(C_{0,q},C_{1},\beta,q)\cdot\left(\Lambda+\|\eta^{2}f^{\beta}(v)\|_{L^{1}(M\times[0,T])}\right),
\end{eqnarray*}
where $\widetilde{C}_{n,k,T}(C_{0,q},C_{1},\beta,q)$ is the constant
depending only on $n,k,T, \beta,q$, $C_{0,q}$, $C_{1}$, and ${\rm
Vol}(M)$. From the definition of $A$ and noting that
$1<\frac{\beta}{\beta-1}\leq2$, one yields
\begin{eqnarray*}
& & \|\eta^{2}f^{\beta}(v)\|_{L^{\gamma/2}(M\times[0,T])}
\\
&\leq&\widetilde{C}_{n,k,T}(C_{0,q},C_{1},\beta,q)\left(\frac{\beta}{\beta-1}\int^{s}_{0}\int_{M_{t}}f^{\beta}(v)\left[2\eta\left(\frac{\partial}{\partial
t}-f'(v)\Delta_{t}\right)\eta\right.\right.
\\
&+&\left.\left(\frac{1}{\beta}\frac{f(v)f''(v)}{f'(v)}+\frac{8\beta^{2}-2\beta+2}{\beta(\beta-1)}f'(v)\right)|\nabla_{t}\eta|^{2}_{g(t)}\right]d\mu(t)dt
\\
&+&\left.\int^{s}_{0}\int_{M_{t}}f^{\beta}(v)\eta^{2}d\mu(t)dt\right)
\\
&\leq&C_{n,k,T}(C_{0,q},C_{1},\beta,q)\int^{s}_{0}\int_{M_{t}}f^{\beta}(v)\left[\eta^{2}+2\eta\left(\frac{\partial}{\partial
t}-f'(v)\Delta_{t}\right)\eta\right. \\
&+&\left.\left(\frac{1}{\beta}\frac{f(v)f''(v)}{f'(v)}+\frac{8\beta^{2}-2\beta+2}{\beta(\beta+1)}f'(v)\right)|\nabla_{t}\eta|^{2}_{g(t)}\right]d\mu(t)dt
\\
&=&C_{n,k,T}(C_{0,q},C_{1},\beta,q)\left\|f^{\beta}(v)\left[\eta^{2}+2\eta\left(\frac{\partial}{\partial
t}-f'(v)\Delta_{t}\right)\eta\right.\right. \\
&+&\left.\left.\left(\frac{1}{\beta}\frac{f(v)f''(v)}{f'(v)}+\frac{8\beta^{2}-2\beta+2}{\beta(\beta-1)}f'(v)\right)|\nabla_{t}\eta|^{2}_{g(t)}\right]\right\|_{L^{1}(M\times[0,T])},
\end{eqnarray*}
which is our required result.
\end{proof}

Taking some special smooth function and using the Moser iteration,
we can prove that the $L^{\infty}$-norm of $v$ over a smaller domain
is bounded by some $L^{\beta}$-norm of $v$ over the whole manifold
$M\times[0,T]$.

\begin{corollary} Suppose that the integers $n$ and $k$ are greater than or equal to $2$. Consider the GMCF
\begin{equation*}
\frac{\partial}{\partial t}F(\cdot,t)=-f(H(\cdot,t))\nu(\cdot,t), \
\ \ 0\leq t\leq T\leq T_{{\rm max}}<\infty.
\end{equation*}
Suppose that $f\in C^{\infty}(\Omega)$ for an open set
$\Omega\subset\mathbb{R}$, and that $v$ is a smooth function on
$M\times[0,T]$ such that its image is contained in $\Omega$.
Consider the differential inequality
\begin{equation}
\left(\frac{\partial}{\partial t}-\Delta_{f,t}\right)v\leq G\cdot
f(v)+f''(v)|\nabla_{t}v|^{2}, \ \ \ v\geq0, \ \ \ G\in
L^{q}(M\times[0,T]).
\end{equation}
Let
\begin{equation*}
C_{0,\infty}=\|f'(v)G\|_{L^{\infty}(M\times[0,T])}, \ \ \
C_{1}=\left(1+\|H\|^{n+k+1}_{L^{n+k+1}(M\times[0,T])}\right)^{\frac{1}{k}},
\end{equation*}
and also let
\begin{equation*}
\gamma=2+\frac{(k+1)^{2}}{k^{2}n}.
\end{equation*}
We denote by $\mathcal{S}$ the set of all functions $f\in
C^{\infty}(\Omega)$, where $\Omega\subset\mathbb{R}$ is the domain
of $f$, satisfying

\begin{itemize}

\item[(i)] $f$ satisfies the differential inequality (4.3),

\item[(ii)] $f'(x)>0$ for all $x\in\Omega$,

\item[(iii)] $f(x)\geq0$ whenever $x\geq0$,

\item[(iv)] $f(H(t))H(t)\geq0$ along the GMCF.

\item[(v)] $f'(v)\geq C_{2}>0$ on $M\times[0,T]$ for some uniform constant $C_{2}$.

\end{itemize}
There exists an uniform constant $C_{n}>0$, depending only on $n$,
such that for any $\beta\geq2$ and $f\in\mathcal{S}$ we have
\begin{equation*}
\|f(v)\|_{L^{\infty}\left(M\times\left[\frac{T}{2},T\right]\right)}\leq
E_{n,k,T}(\beta)\cdot
C^{\frac{1}{\beta}\frac{2}{\gamma-2}}_{1}\cdot\|f(v)\|_{L^{\beta}(M\times[0,T])},
\end{equation*}
where
\begin{equation*}
E_{n,k,T}(\beta)=(D_{n,k,T}C_{n}\beta)^{\frac{1}{\beta}\frac{\gamma}{\gamma-2}}\cdot\left(\frac{\gamma}{2}\right)^{\frac{1}{\beta}\frac{2\gamma}{(\gamma-2)^{2}}}\cdot
4^{\frac{1}{\beta}\frac{\gamma^{2}}{(\gamma-2)^{2}}}.
\end{equation*}
\end{corollary}

\begin{proof} Consider an increasing sequence of times $t_{i}$ defined by
\begin{equation*}
t_{i}=\frac{T}{2}\left(1-\frac{1}{4^{i}}\right), \ \ \
i=0,1,2,\cdots.
\end{equation*}
Consider a sequence of smooth function $\eta_{l}(t)$ satisfying the
following properties
\begin{equation*}
\eta_{i}|_{[t_{i},T]}\equiv1, \ \ \ \eta_{i}|_{[0,t_{i-1}]}\equiv0,
\ \ \ 0\leq\eta\leq 1, \ \ \ |\eta'_{i}|\leq C_{n}4^{i}.
\end{equation*}
For convenience, we denote by $I_{i}$ the interval $[t_{i},T]$.
Since $\|f'(v)G\|_{L^{\infty}(M\times[0,T])}$ exists, letting
$\gamma\to\infty$, we have
\begin{equation*}
\|f^{\beta}(v)\|_{L^{\gamma/2}(M\times I_{i})}\leq [D_{n,k,T}\cdot
C_{n}\cdot 4^{i}]\cdot\beta\cdot
C^{2/\gamma}_{1}\|f^{\beta}(v)\|_{L^{1}(M\times I_{i-1})}.
\end{equation*}
For a moment we put $C=D_{n,k,T}C_{n}$,
$\|\cdot\|_{p,i}=\|\cdot\|_{L^{p}(M\times I_{i})}$,
$\widehat{\gamma}=\gamma/2$, and $w=f(v)$. Hence
\begin{equation*}
\|w^{\beta}\|_{\widehat{\gamma},i}\leq C\beta
C^{1/\widehat{\gamma}}_{1}4^{i}\|w^{\beta}\|_{1,i-1}, \ \ \
\|w\|_{\beta\widehat{\gamma},i}\leq
C^{\frac{1}{\beta}}\beta^{\frac{1}{\beta}}C_{1}^{1/\beta\widehat{\gamma}}4^{\frac{i}{\beta}}\|w\|_{\beta,i-1}.
\end{equation*}
Replacing $\beta$ by $\widehat{\gamma}^{i-1}\beta$, we derive
\begin{eqnarray*}
\|w\|_{\beta\widehat{\gamma}^{m},m}&\leq&
C^{\sum^{m-1}_{i=0}\frac{1}{\beta\widehat{\gamma}^{i}}}\cdot\prod^{m-1}_{i=0}(\beta\widehat{\gamma}^{i})^{\frac{1}{\beta\widehat{\gamma}^{i}}}\cdot
C_{1}^{\sum^{m}_{i=1}\frac{1}{\beta\widehat{\gamma}^{i}}}\cdot
4^{\sum^{m-1}_{i=0}\frac{i+1}{\beta\widehat{\gamma}^{i}}}\|w\|_{\beta,0},
\\
&=&(C\beta)^{\frac{1}{\beta}\sum^{m-1}_{i=0}\frac{1}{\widehat{\gamma}^{i}}}\cdot
C_{1}^{\frac{1}{\beta}\sum^{m}_{i=1}\frac{1}{\widehat{\gamma}^{i}}}\cdot\widehat{\gamma}^{\frac{1}{\beta}\sum^{m-1}_{i=0}\frac{i}{\widehat{\gamma}^{i}}}\cdot
4^{\frac{\widehat{\gamma}}{\beta}\sum^{m}_{i=0}\frac{i}{\widehat{\gamma}^{i}}}\|w\|_{\beta,0}.
\end{eqnarray*}
From the elementary facts on power series we have
\begin{equation*}
\sum^{\infty}_{i=0}\frac{1}{\widehat{\gamma}^{i}}=\frac{\widehat{\gamma}}{\widehat{\gamma}-1},
\ \ \
\sum^{\infty}_{i=0}\frac{i}{\widehat{\gamma}^{i}}=\frac{\widehat{\gamma}}{(\widehat{\gamma}-1)^{2}},
\end{equation*}
consequently,
\begin{eqnarray*}
\|w\|_{\infty,\infty}&\leq&(C\beta)^{\frac{1}{\beta}\frac{\widehat{\gamma}}{\widehat{\gamma}-1}}\cdot
C_{1}^{\frac{1}{\beta}\frac{1}{\widehat{\gamma}-1}}\cdot\widehat{\gamma}^{\frac{1}{\beta}\frac{\widehat{\gamma}}{(\widehat{\gamma}-1)^{2}}}\cdot
4^{\frac{\widehat{\gamma}}{\beta}\frac{\widehat{\gamma}}{(\widehat{\gamma}-1)^{2}}}\|w\|_{\beta,0},
\\
&=&E_{n,k,T}(\beta)\cdot
C^{\frac{1}{\beta}\frac{2}{\gamma-2}}_{1}\cdot\|w\|_{\beta,0}.
\end{eqnarray*}
Since $I_{\infty}=[T/2,T]$ and $I_{0}=[0,T]$, the corollary
immediately follows.
\end{proof}

\begin{corollary} Suppose that the integers $n$ and $k$ are greater than or equal to $2$ and that $n+1\geq k$. Consider the
$H^{k}$ mean curvature flow
\begin{equation*}
\frac{\partial}{\partial t}F(\cdot,t)=-H^{k}(\cdot,t)\nu(\cdot,t), \
\ \ 0\leq t\leq T\leq T_{{\rm max}}<\infty.
\end{equation*}
If
\begin{equation*}
H(t)\geq\left(\frac{C_{2}}{k}\right)^{\frac{1}{k-1}}>0, \ \ \
\|kH^{k-1}(t)A^{2}(t)\|_{L^{\infty}(M\times[0,T])}<\infty,
\end{equation*}
along the $H^{k}$ mean curvature flow for some uniform constant
$C_{2}>0$, then there exists an uniform constant $C_{n}$, depending
only on $n$, such that
\begin{eqnarray*}
\|H(t)\|_{L^{\infty}\left(M\times\left[\frac{T}{2},T\right]\right)}&\leq&
E^{1/k}_{n,k,T}\left(\frac{n+k+1}{k}\right)\left(1+\|H(t)\|^{n+k+1}_{L^{n+k+1}(M\times[0,T])}\right)^{\frac{2}{\gamma-2}\frac{1}{n+k+1}}
\\
&\cdot&\|H(t)\|_{L^{n+k+1}(M\times[0,T])}, \\
&\leq& F_{n,k,T_{{\rm max}}}\cdot\|H(t)\|_{L^{n+k+1}(M\times[0,T])},
\end{eqnarray*}
where
\begin{equation*}
F_{n,k,T_{{\rm max}}}=E^{1/k}_{n,k,T_{{\rm
max}}}\left(\frac{n+k+1}{k}\right)\left(1+\|H(t)\|^{n+k+1}_{L^{n+k+1}(M\times[0,T_{{\rm
max}}))}\right)^{\frac{2}{\gamma-2}\frac{1}{n+k+1}}.
\end{equation*}
\end{corollary}

\begin{proof} Let $f(x)=x^{k}: \mathbb{R}_{+}\to\mathbb{R}$. From the
evolution equation for $H(t)$,
\begin{equation*}
\left(\frac{\partial}{\partial
t}-\Delta_{f,t}\right)H(t)=|A(t)|^{2}_{g(t)}\cdot
f(H(t))+f''(H(t))|\nabla_{t}H(t)|^{2}_{g(t)},
\end{equation*}
we know that $G(t)=|A(t)|^{2}_{g(t)}$ and all conditions in
Corollary 4.3 are satisfied. Hence there is an uniform constant
$C_{n}$ such that
\begin{equation*}
\|H^{k}(t)\|_{L^{\infty}\left(M\times\left[\frac{T}{2},T\right]\right)}\leq
E_{n,k,T}(\beta)C^{\frac{1}{\beta}\frac{2}{\gamma-2}}_{1}\|H^{k}(t)\|_{L^{\beta}(M\times[0,T])}
\end{equation*}
Taking $k$th root on both sides, we have
\begin{equation*}
\|H(t)\|_{L^{\infty}\left(M\times\left[\frac{T}{2},T\right]\right)}\leq
E^{1/k}_{n,k,T}(\beta)C^{\frac{2}{\gamma-2}\frac{1}{k\beta}}_{1}\|H(t)\|_{L^{k\beta}(M\times[0,T])}.
\end{equation*}
If we chose $\beta=\frac{n+k+1}{k}\geq2$, then it follows that
\begin{equation*}
\|H(t)\|_{L^{\infty}\left(M\times\left[\frac{T}{2},T\right]\right)}\leq
E^{1/k}_{n,k,T}\left(\frac{n+k+1}{k}\right)\cdot
C^{\frac{2}{\gamma-2}\frac{1}{n+k+1}}_{1}\|H(t)\|_{L^{n+k+1}(M\times[0,T])}.
\end{equation*}
By the definition of $E_{n,k,T}$ and $C_{1}$, the required
inequality immediately follows.
\end{proof}

\begin{remark} When $k=1$, the assumption $n+1\geq k$ is obvious.
but for $k\geq2$, this assumption is necessarily needed in our
proof. In the forthcoming paper we may remove this condition.
\end{remark}

\section{Proof of the main theorem and further remarks}

The proof of our main theorem is similar to that given in
\cite{Xu-Ye-Zhao}, hence in this section we only give a sketch
proof. From H\"older's inequality, it is sufficient to proof the
theorem for $\alpha=n+k+1$. Note that the quantity
$\|H\|_{L^{\alpha}(M\times[0,T])}$ is invariant under the rescaling
of the mean curvature flow
\begin{equation}
\widetilde{F}(p,t)=Q^{\frac{1}{k+1}} \cdot
F\left(p,\frac{t}{Q}\right)
\end{equation}
for $Q>0$.

Suppose that the solution can not be extended over $T_{{\rm max}}$.
Hence we know that $|A(t)|_{g(t)}$ is unbounded as $t\to T_{{\rm
max}}$. Let $\lambda_{i}(i=1,\cdots,n)$ denote the principle
curvatures. Then
\begin{equation*}
|A(t)|^{2}_{g(t)}=\sum^{n}_{i=1}\lambda^{2}_{i}\leq\left(\sum^{n}_{i=1}\lambda_{i}\right)^{2}=H^{2}(t).
\end{equation*}
Thus, $H^{k+1}(x,t)$ is also unbounded. We can chose a sequence of
times $\{t^{(i)}\}^{\infty}_{i=1}$ with
$\lim_{t\to\infty}t^{(i)}=T_{{\rm max}}$ and a sequence of points
$\{x^{(i)}\}^{\infty}_{i=1}$ such that
\begin{equation*}
Q^{(i)}=H^{k+1}(x^{(i)},t^{(i)})=\max_{(x,t)\in
M\times[0,t^{(i)})}H^{k+1}(x,t)\to\infty.
\end{equation*}
Therefore there exists an integer $i_{0}$ such that
$(Q^{(i)})^{\frac{2}{k+1}}t^{(i)}\geq1$ for any $i\geq i_{0}$.
Define
\begin{equation*}
F^{(i)}(x,t)=(Q^{(i)})^{\frac{1}{k+1}}F\left(x,\frac{t-1}{(Q^{(i)})^{\frac{2}{k+1}}}+t^{(i)}\right),
\ \ \ i\geq i_{0}, \ \ \ t\in[0,1].
\end{equation*}
Then a simple calculus shows that
\begin{eqnarray*}
g^{(i)}(x,t)&=&(Q^{(i)})^{\frac{2}{k+1}}g\left(x,\frac{t-1}{(Q^{(i)})^{\frac{2}{k+1}}}+t^{(i)}\right),
\\
h^{(i)}_{pq}(x,t)&=&(Q^{(i)})^{\frac{1}{k+1}}h_{pq}\left(x,\frac{t-1}{(Q^{(i)})^{\frac{2}{k+1}}}+t^{(i)}\right),
\\
H^{(i)}(x,t)&=&(Q^{(i)})^{-\frac{1}{k+1}}H\left(x,\frac{t-1}{(Q^{(i)})^{\frac{2}{k+1}}}+t^{(i)}\right),
\end{eqnarray*}
where $g^{(i)}, h^{(i)}_{pq}$, and $H^{(i)}$ are the corresponding
induced metric, second fundamental forms, and mean curvature,
respectively. From the definition of $Q^{(i)}$ we must that
\begin{equation*}
(H^{(i)}(x,t))^{k+1}\leq1, \ \ \ 0\leq h^{(i)}_{pq}(x,t)\leq 1, \ \
\ (x,t)\in M\times[0,1].
\end{equation*}
As in \cite{Xu-Ye-Zhao}, we can find a subsequence of
$\{M,g^{(i)}(t),F^{(i)}(t),x^{(i)}\}$, $t\in[0,1]$, converges to a
Riemannian manifold $(\widetilde{M},\widetilde{g}(t),
\widetilde{F}(t), \widetilde{x})$, where $\widetilde{F}(t):
\widetilde{M}\to\mathbb{R}^{n+1}$ is an immersion. Since
$(H^{(i)}(x,t))^{k+1}\leq1$ on $M\times[0,1]$ for all $i\geq i_{0}$,
it follows that $k(H^{(i)}(x,t))^{k-1}(A^{(i)}(x,t))^{2}$ is also
bounded by $1$ on $M\times[0,1]$ and any $i\geq i_{0}$.
Consequently, we have, using Corollary 4.4,
\begin{equation*}
\max_{(x,t)\in
M^{(i)}\times\left[\frac{1}{2},1\right]}(H^{(i)}(x,t))^{k+1}\leq
C\left(\int^{1}_{0}\int_{M^{(i)}}|H^{(i)}(x,t)|^{n+k+1}d\mu_{g^{(i)}}(t)dt\right)^{\frac{k+1}{n+k+1}}
\end{equation*}
for some uniform constant $C$. Since the quantity
$\|H\|^{n+k+1}_{L^{n+k+1}(M\times[0,T])}$ in invariant under the
rescaling of the $H^{k}$ mean curvature flow
$Q^{\frac{1}{k+1}}F(\cdot,\frac{t}{Q})$, one has
\begin{equation*}
\max_{(x,t)\in\widetilde{M}\times\left[\frac{1}{2},1\right]}\widetilde{H}^{k+1}(x,t)=\lim_{i\to\infty}
\max_{(x,t)\in
M^{(i)}\times\left[\frac{1}{2},1\right]}(H^{(i)}(x,t))^{k+1}\leq0.
\end{equation*}
On the other hand, by our construction, we must have
\begin{equation*}
\widetilde{H}^{k+1}(\widetilde{x},1)=\lim_{i\to\infty}(H^{(i)}(x^{(i)},1))^{k+1}=1.
\end{equation*}
This contradiction implies that the solution of the $H^{k}$ mean
curvature flow can be extended over $T_{{\rm max}}$.

\begin{remark} A natural question is to weaken the curvature
condition on $M$. The main reason why we assume that the mean
curvature of $M$ has positive lower bound, comes from the term
$H^{k-1}$; in the linear case $k=1$, this term must be a constant,
but for the nonlinear case $k\geq2$, we should impose some curvature
conditions on $M$ to guarantee the boundedness of such term.

Our method mainly depends on \cite{Le-Sesum}, therefore, we may find
other approaches to deal with the nonlinear case and to remove the
positivity lower bound of the mean curvature on $M$. These will be
treated with in the forthcoming paper \cite{Ostergaard-Li}.
\end{remark}

\bibliographystyle{amsplain}

\begin{thebibliography}{10}

\bibitem{Cooper} A. A., Cooper, \textit{Mean curvature blow up in
mean curvature flow}, preprint, arXiv: math.DG/0902.4282.

\bibitem{Grafakos} L. Grafakos, \textit{Classical Fourier Analysis}, 2nd ed., Graduate Texts in Mathematics, Springer, 2008.

\bibitem{Huisken} G. Huisken, \textit{Flow by mean curvature of convex surfaces into spheres}, J. Diffeential Geom. \textbf{20} (1984), no.1, 237--266.

\bibitem{PLax} P. Lax, \textit{Formuation and propagation of shock
waves, the role of entropy: A history of shock waves and its future
development}, First Ahlfors lecture (Public), organized by
Shing-Tung Yau and Horng-Tzer Yau, Science Center Lecture Hall D,
Harvard, 4:30 pm--5:30 pm, Oct.5., 2009.

\bibitem{Le-Sesum} N. Q. Le, N. Sesum, \textit{On the extension of the mean curvature flow}, preprint, arXiv: math.DG/0905.0936v2.

\bibitem{Michael-Simon} G. Michael, L. Simon, \textit{Sobolev and mean-value inequalities on generalized submanifolds of $\mathbb{R}^{n}$}, Comm. Pure Appl. Math. \textbf{26}(1973), 361-379.

\bibitem{Ostergaard-Li} A. Ostergaard, Y. Li, \textit{On an extension of
the $H^{k}$ mean curvature flow II}, in preparation.

\bibitem{Smoczyk} K, Smoczyk, \textit{Harnack inequalities for
curvature flows depending on mean curvature}, New York J. Math.
\textbf{3} (1997), 103--118.

\bibitem{Xu-Ye-Zhao} H. W. Xu, F. Ye, E. T. Zhao, \textit{Extend mean
curvature flow with finite integral curvature}, preprint, arXiv:
math.DG/0905.1167v1.


\end{thebibliography}

\end{document}